\documentclass[a4paper,leqno]{amsart}
\usepackage{latexsym}
\usepackage[english]{babel}
\usepackage{fancyhdr}
\usepackage[mathscr]{eucal}
\usepackage{amsmath}
\usepackage{mathrsfs}
\usepackage{amsthm}
\usepackage{amsfonts}
\usepackage{amssymb}
\usepackage{amscd}
\usepackage{bbm}
\usepackage{graphicx}
\usepackage{graphics}
\usepackage{latexsym}
\usepackage{color}

\usepackage{pifont}
\usepackage{booktabs} 

\newcommand{\ud}{\mathrm{d}}

\newcommand{\cH}{\mathcal{H}}

 \newtheorem{thm}{Theorem}[section]
 \newtheorem{cor}[thm]{Corollary}
 \newtheorem{lem}[thm]{Lemma}
 \newtheorem{prop}[thm]{Proposition}
 \theoremstyle{definition}
 \newtheorem{defn}[thm]{Definition}
 \newtheorem{rem}[thm]{Remark}
 \newtheorem{ex}{Example}
 \numberwithin{equation}{section}

\begin{document}

\title[Krylov solvability of unbounded inverse linear problems]
 {Krylov solvability of unbounded \\ inverse linear problems}

\author{No\`e Angelo Caruso}

\address{%
International School for Advanced Studies\\
via Bonomea 265, I-34136 Trieste (ITALY)\\
and Gran Sasso Science Institute\\
Viale F.~Crispi 7, I-67100 L'Aquila (ITALY)}

\email{noe.caruso@gssi.it}

\author{Alessandro Michelangeli}
\address{Institute for Applied Mathematics \\
and Hausdorff Center for Mathematics\\
University of Bonn\\
Endenicher Allee 60, D-53115 Bonn (GERMANY)}
\email{michelangeli@iam.uni-bonn.de}



\date{\today}
\makeatletter{\renewcommand*{\@makefnmark}{}
\footnotetext{Integral Equations and Operator Theory.}\makeatother}

\begin{abstract}
 The abstract issue of `Krylov solvability' is extensively discussed for the inverse problem $Af=g$ where $A$ is a (possibly unbounded) linear operator on an infinite-dimensional Hilbert space, and $g$ is a datum in the range of $A$. The question consists of whether the solution $f$ can be approximated in the Hilbert norm by finite linear combinations of $g,Ag,A^2g,\dots$, and whether solutions of this sort exist and are unique. After revisiting the known picture when $A$ is bounded, we study the general case of a densely defined and closed $A$. Intrinsic operator-theoretic mechanisms are identified that guarantee or prevent Krylov solvability, with new features arising due to the unboundedness. Such mechanisms are checked in the self-adjoint case, where Krylov solvability is also proved by conjugate-gradient-based techniques. 
\end{abstract}

\maketitle

\section{Introduction and set-up of the problem}\label{intro}

The question of `\emph{Krylov solvability}' of an inverse linear problem is an operator-theoretic question, with deep-rooted implications in numerics and scientific computing among others, that in fairly abstract terms is formulated as follows.

A linear operator $A$ acting on a real or complex Hilbert space $\cH$, and a vector $g\in\cH$ are given such that $A$ is closed and densely or everywhere defined on $\cH$, and $g$ is an $A$-smooth vector in the range of $A$, i.e.,
\begin{equation}\label{eq:assumption-g}
 g\;\in\;\mathrm{ran} A\cap C^\infty(A)
\end{equation}
where $C^\infty(A)$ is the space of elements of $\cH$ simultaneously belonging to all the domains of the natural powers of $A$,
\begin{equation}\label{eq:gsmooth}
 C^\infty(A)\;:=\;\bigcap_{k\in\mathbb{N}}\mathcal{D}(A^k)\,.
\end{equation}
Clearly $A$-smoothness is an automatic condition if $A$ is bounded. Associated with $A$ and $g$ one has the `\emph{Krylov subspace}'
\begin{equation}\label{eq:defKrylov}
  \mathcal{K}(A,g) \;:=\;  \mathrm{span} \{ A^kg\,|\,k\in\mathbb{N}_0\}\,,
\end{equation}
as well as the inverse linear problem induced by $A$ with datum $g$, namely the problem of finding solution(s) $f\in\mathcal{D}(A)$ such that
\begin{equation}\label{eq:invLP}
 Af\;=\;g\,.
\end{equation}
The problem \eqref{eq:invLP} is said to be `\emph{Krylov-solvable}' if for some solution $f$ one has 
\begin{equation}
 f\;\in\;\overline{\mathcal{K}(A,g)}\,,
\end{equation}
in which case $f$ is also referred to as a `\emph{Krylov solution}'.

In short, Krylov solvability for the problem \eqref{eq:invLP} is the possibility of having solution(s) $f$ for which there are approximants, in the Hilbert norm, consisting of finite linear combinations of vectors $g,Ag,A^2g,A^3g,\dots$.

This explains the deep conceptual relevance of Krylov solvability in scientific computing: knowing a priori whether or not an inverse problem is Krylov-solvable allows one to decide whether to treat the problem numerically by means of one of the vast class of so-called Krylov-subspace methods \cite{Saad-2003_IterativeMethods,Liesen-Strakos-2003}, searching for approximants to the exact solution(s) over the span of \emph{explicit} trial vectors $g,Ag,A^2g,A^3g,\dots$.

In fact, Krylov subspace methods are efficient numerical schemes for finite-dimensional inverse linear problems, even counted among the `Top 10 Algorithms' of the 20th century \cite{Dongarra-Sullivan-Best10-2000,Cipra-SIAM-News}, a framework that is by now classical and deeply understood (see, e.g., the monographs \cite{Saad-2003_IterativeMethods,Liesen-Strakos-2003} or also \cite{Saad-1981}), and they are naturally exported to the infinite-dimensional case ($\dim\cH=\infty$), although the latter is less systematically studied and is better understood through special sub-classes of interest \cite{Karush-1952,Daniel-1967,Kammerer-Nashed-1972,Nemirovskiy-Polyak-1985,Winther-1980,Herzog-Ekkehard-2015,CMN-2018_Krylov-solvability-bdd,CM-Nemi-unbdd-2019}. 
Of course we refer here to the circumstance  when \eqref{eq:invLP} is genuinely infinite-dimensional, meaning that not only $\dim\cH=\infty$, but also (see, e.g., \cite[Sect.~1.4]{schmu_unbdd_sa}) that $A$ is \emph{not} reduced to $A=A_1\oplus A_2$ by an orthogonal direct sum decomposition $\cH=\cH_1\oplus\cH_2$ with $\dim\cH_1<\infty$, $\dim\cH_2=\infty$, and $A_2=\mathbb{O}$ (for otherwise the effective problem would deal with a finite matrix $A_1$).

Clearly, Krylov solvability is non-trivial whenever $\mathcal{K}(A,g)$ admits a non-trivial orthogonal complement in $\cH$.

Thus, for example, for the (everywhere defined and bounded) multiplication operator $M:L^2[1,2]\to L^2[1,2]$, $\psi\mapsto x\psi$, and for the function $g=\mathbf{1}$ (the constant function with value 1), $\mathcal{K}(M,\mathbf{1})$ is the space of polynomials on $[1,2]$, hence it is \emph{dense} in $L^2[1,2]$: the solution to $Mf=\mathbf{1}$, which is explicitly $f(x)=x^{-1}$, obviously belongs to $\overline{\mathcal{K}(M,\mathbf{1})}$. On the other hand, for the (everywhere defined and bounded) right-shift operator $R$ on $\ell^2(\mathbb{Z})$ defined with respect to the canonical orthonormal basis $(e_n)_{n\in\mathbb{Z}}$ by $e_n\mapsto e_{n+1}$, and for the vector $g=e_{1}$, one has $\overline{\mathcal{K}(R,e_{1})}=\mathrm{span}\{e_0,e_{-1},e_{-2},\dots\}^\perp$: the problem $Rf=e_1$ is solved by $f=e_0$ which does not belong to $\overline{\mathcal{K}(R,e_{1})}$.

If the operator $A$ is non-injective and hence the inverse problem \eqref{eq:invLP} admits a multiplicity of solutions, it may also well happen that some are Krylov solutions and others are not (Example \eqref{example:Krylov_sol}(iv) below).

These considerations suggest that the general issue of Krylov solvability can be posed from various specific perspectives, such as:
\begin{itemize}
 \item given \emph{explicit} $A$ and $g$, to decide about the existence or the uniqueness of a Krylov solution to the inverse problem \eqref{eq:invLP};
 \item to identify \emph{classes of operators} $A$ that induce Krylov-solvable problems  \eqref{eq:invLP} irrespective of choice of the datum $g$ (as long as $g$ satisfies the basic condition \eqref{eq:assumption-g});
 \item to qualify `intrinsic' mechanisms of Krylov solvability through general conditions on $A$ and $g$.
\end{itemize}

At present the above conceptual programme appears to be only partially developed.

Given the iterative nature of Krylov subspace methods, the largest part of the related literature is mainly concerned with the fundamental issue of \emph{convergence} (an ample overview of which can be found, for instance, in the monographs \cite{Nevanlinna-1993-converg-iterat-book,Engl-Hanke-Neubauer-1996,Saad-2003_IterativeMethods}) of the Krylov approximants to a solution $f$. It is clear, however, as shown by the above simple example on the inverse problem $Rf=e_1$, that the question of \emph{Krylov solvability} is equally fundamental to decide when to attack an inverse problem by means of computational methods that make use of Krylov approximants.

Motivated by the implications in numerical analysis as well as by abstract operator-theoretical interest, in the recent work \cite{CMN-2018_Krylov-solvability-bdd} in collaboration with P.~Novati we discussed the question of Krylov solvability when in the inverse problem \eqref{eq:invLP} $\dim\cH=\infty$ and $A$ is \emph{bounded}, with a special focus on normal and, in particular, on self-adjoint operators. Certain operator-theoretic mechanisms were identified (that we called `Krylov reducibility', and `triviality of the Krylov intersection', among others) which account for how, at a `structural' level, Krylov solvability occurs and when the Krylov solution is unique. Section \ref{sec:bounded} of the present work reviews those findings that are relevant for the subsequent discussion.

Along a parallel route, in \cite{CM-Nemi-unbdd-2019} we studied the convergence of a popular Krylov subspace algorithm for the inverse problem \eqref{eq:invLP}, the so-called method of conjugate gradient, in the generalised setting when $A$ is \emph{unbounded}. In view of the conceptual programme above, \cite{CM-Nemi-unbdd-2019} can be regarded as a first step to study the Krylov solvability of \eqref{eq:invLP} in the unbounded case -- and section \ref{sec:cg} here accounts for that perspective -- however with the two-fold limitation that $A$ has to be self-adjoint and non-negative (as required in conjugate gradient methods), and that Krylov solvability emerges only as a by-product result with no explicit insight on the operator-theoretic mechanism for it.

The present work aims at pushing our programme further by discussing Krylov solvability for a fairly general class of unbounded $A$s and with a focus on the same structural mechanisms previously identified in the bounded case.

In section \ref{sec:selfadj_skewadj} we present our first result that answers (in the affirmative) the question of Krylov solvability when $A$ is generically (unbounded and) self-adjoint or skew-adjoint.

Then in section \ref{sec:general_unbdd} we proceed on to the general case when $A$ is densely defined and closed on $\cH$. Here we identify new obstructions in the issue of Krylov solvability, which are not present in the bounded case. A most serious one is the somewhat counterintuitive phenomenon of `Krylov escape', namely the possibility that vectors of $\overline{\mathcal{K}(A,g)}$ that also belong to the domain of $A$ are mapped by $A$ \emph{outside} of $\overline{\mathcal{K}(A,g)}$, whereas obviously $A\mathcal{K}(A,g)\subset\mathcal{K}(A,g)$. From a perspective that in fact we are not carrying over here, one might observe that the possibility of Krylov escape adds further complication to the unbounded operator counterpart of the celebrated invariant subspace problem \cite{Yadav2005_InvSubsp}, at least when $g\neq 0$ and $g$ is not a cyclic vector for $A$, hence $\overline{\mathcal{K}(A,g)}$ is a proper closed subspace of $\cH$.

In section \ref{sec:general_unbdd} we also determine that if the closures of $\mathcal{K}(A,g)$ in the Hilbert space norm and in the stronger $A$-graph norm are the same (up to intersection with $\mathcal{D}(A)$), an occurrence that we named `Krylov-core condition', then Krylov escape is actually prevented.

This leads us to section \ref{sec:genunbdd}, where we generalise to the case of densely defined and closed $A$s the previously known picture of Krylov solvability when $A$ was bounded. In particular, we demonstrate that under assumptions like the Krylov-core condition (and, more generally, lack of Krylov escape) the intrinsic mechanisms of Krylov reducibility and triviality of the Krylov intersection play a completely analogous role as compared to the bounded case.

Last, in section \ref{sec:selfadj_revisited} we re-consider the (unbounded) self-adjoint scenario, that from the practical point of view is already solved in section \ref{sec:selfadj_skewadj}, investigating Krylov solvability from the perspective of the abstract operator-theoretic mechanisms mentioned above. Noticeably, this is also a perspective that rises up interesting open questions. Indeed, whereas we can prove that self-adjoint operators do satisfy the Krylov-core condition and are Krylov-reducible for a distinguished dense set of $A$-smooth vectors $g$'s, and that for the same choice of $g$ the subspace $\overline{\mathcal{K}(A,g)}$ is naturally isomorphic to $L^2(\mathbb{R},\ud\mu_g^{(A)})$ (here $\mu_g^{(A)}$ is the scalar spectral measure), we cannot decide whether Krylov escape is prevented for \emph{any} self-adjoint $A$ and $A$-smooth $g$ (which is remarkable, as by other means we know that $Af=g$ is Krylov solvable at least for a dense set of smooth vectors $g$). This certainly indicates a future direction of investigation.

\medskip

\textbf{General notation.} Besides further notation that will be declared in due time, we shall keep the following convention. $\cH$ denotes a complex Hilbert space with norm $\|\cdot\|_{\cH}$ and scalar product $\langle\cdot,\cdot\rangle$, anti-linear in the first entry and linear in the second. Bounded operators on $\cH$ shall be tacitly understood as linear and everywhere defined: they naturally form a space, denoted with $\mathcal{B}(\cH)$, with Banach norm $\|\cdot\|_{\mathrm{op}}$, the customary operator norm.
$\mathbbm{1}$ and $\mathbbm{O}$ denote, respectively, the identity and the zero operator, meant as finite matrices or infinite-dimensional operators depending on the context. An upper bar denotes the complex conjugate $\overline{z}$ when $z\in\mathbb{C}$, and the norm closure $\overline{\mathcal{V}}$ of the span of the vectors in $\mathcal{V}$ when $\mathcal{V}$ is a subset of $\cH$. For $\psi,\varphi\in\cH$, by $|\psi\rangle\langle\psi|$ and $|\psi\rangle\langle\varphi|$ we shall denote the $\cH\to\cH$ rank-one maps acting respectively as $f\mapsto \langle \psi, f\rangle\,\psi$ and $f\mapsto \langle \varphi, f\rangle\,\psi$ on generic $f\in\cH$. For identities such as $\psi(x)=\varphi(x)$ in $L^2$-spaces, the standard `for almost every $x$' declaration will be tacitly understood.

\section{The bounded case}\label{sec:bounded}

Krylov solvability of the inverse problem \eqref{eq:invLP} when $\dim\cH=\infty$, when $A$ is (everywhere defined, and) bounded and $g\in\mathrm{ran} A$, appears to manifest or to fail to hold in a variety of situations.

\begin{ex}\label{example:Krylov_sol}~
\begin{itemize}
 \item[(i)] The multiplication operator $M_z:L^2(\Omega)\to L^2(\Omega)$, $f\mapsto zf$, where $\Omega\subset\mathbb{C}$ is a bounded open subset separated from the origin, say, $\Omega=\{z\in\mathbb{C}\,|\,|z-2|<1\}$, is a normal bounded bijection on $L^2(\Omega)$, and the solution to $M_zf=g$ for given $g\in L^2(\Omega)$ is the function $f(z)=z^{-1}g(z)$.  Moreover, $\mathcal{K}(M_z,g)=\{p\,g\,|\,p\textrm{ a polynomial in $z$ on }\Omega\}$. One can see that $f\in\overline{\mathcal{K}(M_z,g)}$ and hence the problem $M_zf=g$ is Krylov-solvable. Indeed, $\Omega\ni z\mapsto z^{-1}$ is holomorphic and hence is realised by a uniformly convergent power series (e.g., the Taylor expansion of $z^{-1}$ around $z=2$). If $(p_n)_n$ is such a sequence of polynomial approximants, then
 \[
  \begin{split}
   \|f-p_n g\|_{L^2(\Omega)}\;&=\;\|(z^{-1}-p_n)g\|_{L^2(\Omega)} \\
   &\leqslant\;\|z^{-1}-p_n\|_{L^\infty(\Omega)}\|g\|_{L^2(\Omega)}\;\xrightarrow[]{n\to\infty}\;0\,.
  \end{split}
 \]
 \item[(ii)] The left-shift operator $L$ on $\ell^2(\mathbb{N}_0)$, defined as usual on the canonical basis $(e_n)_{n\in\mathbb{N}_0}$ by $Le_{n+1}=e_n$, is bounded, not injective, and with range $\mathrm{ran}L=\ell^2(\mathbb{N}_0)$. A solution to $Lf=g$ with $g:=\sum_{n\in\mathbb{N}_0}\frac{1}{n!}e_n$ is $f=\sum_{n\in\mathbb{N}_0}\frac{1}{n!}e_{n+1}$. Moreover, $\mathcal{K}(L,g)$ is \emph{dense} in $\ell^2(\mathbb{N}_0)$ and therefore $f$ is a Krylov solution. To see the density of $\mathcal{K}(L,g)$: the vector $e_0$ belongs to $\overline{\mathcal{K}(L,g)}$ because
 \[
  \begin{split}
   \|k!\, L^k g-e_0\|_{\ell^2}^2\;&=\;\|\textstyle(1,\frac{1}{k+1},\frac{1}{(k+2)(k+1)},\cdots)-(1,0,0,\dots)\|_{\ell^2}^2 \\
   &=\;\sum_{n=1}^\infty\Big(\frac{k!}{(n+k)!}\Big)^2\;\xrightarrow[]{\;k\to\infty\;}\;0\,.
  \end{split}
 \]
 As a consequence, $ (0,\textstyle\frac{1}{k!},\frac{1}{(k+1)!},\frac{1}{(k+2)!},\cdots)=L^{k-1}g-(k-1)!\,e_0\in\overline{\mathcal{K}(L,g)}$,
therefore the vector $e_1$ too belongs to $\overline{\mathcal{K}(L,g)}$, because
\[
 \|k!\,(L^{k-1}g-(k-1)!\,e_0)-e_1\|_{\ell^2}^2\;=\;\sum_{n=1}^\infty\Big(\frac{k!}{(n+k)!}\Big)^2\;\xrightarrow[]{\;k\to\infty\;}\;0\,.
\]
 Repeating inductively the above two-step argument proves that any $e_n\in\overline{\mathcal{K}(L,g)}$, whence the cyclicity of $g$. 
  \item[(iii)] The right-shift operator on $\ell^2(\mathbb{Z})$,  
  \begin{equation}\label{eq:compactRshift-Z}
 \mathcal{R}\;=\;\sum_{n\in\mathbb{Z}}|e_{n+1}\rangle\langle e_n|\,,
\end{equation}
is a normal, bounded bijection, and the solution to $\mathcal{R}f = e_2$ is $f=e_1$. However, $f$ is not a Krylov solution, for $\overline{\mathcal{K}(\mathcal{R},e_2)}=\overline{\mathrm{span}\{e_2,e_3,\dots\}}$. The problem $\mathcal{R}f=e_2$ is not Krylov-solvable.
  \item[(iv)] Let $A$ be a bounded injective operator on a Hilbert space $\cH$ with cyclic vector $g\in\mathrm{ran}A$ and let $\varphi_0\in\cH\setminus\{0\}$. Let $f\in\cH$ be the solution to $Af=g$. The operator $\widetilde{A}:=A\oplus |\varphi_0\rangle\langle\varphi_0|$ on $\widetilde{\cH}:=\cH\oplus\cH$ is bounded. One solution to $\widetilde{A}\widetilde{f}=\widetilde{g}:=g\oplus 0$ is $\widetilde{f}=f\oplus 0$ and $\widetilde{f}\in\cH\oplus\{0\}=\overline{\mathcal{K}(\widetilde{A},\widetilde{g})}$. Another solution is $\widetilde{f}_\xi=f\oplus\xi$, where $\xi\in\cH\setminus\{0\}$ and $\xi\perp\varphi_0$. Clearly, $\widetilde{f}_\xi\notin\overline{\mathcal{K}(\widetilde{A},\widetilde{g})}$.
  \item[(v)] If $V$ is the Volterra operator on $L^2[0, 1]$ and $g(x) = \frac{1}{2}x^2$, then $f(x) = x$ is the unique solution to $Vf=g$. On the other hand, $\mathcal{K}(V,g)$ is spanned by the monomials $x^2,x^3,x^4,\dots$, whence 
  \[
  \mathcal{K}(V,g) \;=\; \{x^2p(x)\,|\, p\textrm{ is a polynomial on } [0,1]\}\,. 
  \]
  Therefore $f\notin \mathcal{K}(V,g)$, because $f(x)=x^2\cdot\frac{1}{x}$ and $\frac{1}{x}\notin L^2[0,1]$. Yet, $f\in \overline{\mathcal{K}(V,g)}$, because in fact $\mathcal{K}(V,g)$ is \emph{dense} in $L^2[0,1]$.
  Indeed, if $h\in\mathcal{K}(V,g)^\perp$, then $0=\int_0^1\overline{h(x)}\,x^2p(x)\,\ud x$ for any polynomial $p$; the $L^2$-density of polynomials on $[0,1]$ implies necessarily that $x^2h=0$, whence also $h=0$; this proves that $\mathcal{K}(V,g)^\perp=\{0\}$ and hence $\overline{\mathcal{K}(V,g)}=L^2[0,1]$.
\end{itemize}
\end{ex}

Example \ref{example:Krylov_sol}(iii) shows, in particular, that even apparently stringent assumptions on $A$ such as the simultaneous occurrence of normality, injectivity, and bounded everywhere defined inverse do \emph{not} ensure, in general, that the solution $f$ to $Af=g$, for given $g\in\cH$, is a Krylov solution.

A partial yet fairly informative comprehension of the general bounded scenario was recently reached in the work \cite{CMN-2018_Krylov-solvability-bdd}, where it was shown that Krylov solvability is intrinsically related with certain operator-theoretic mechanisms that we briefly review here.

\begin{defn}\label{def:Kreduc-Kintersec-bounded}
 For a given Hilbert space $\cH$ let $A\in\mathcal{B}(\cH)$ and $g\in\mathrm{ran} A$.
 \begin{itemize}
  \item[(i)] The orthogonal decomposition
  \begin{equation}\label{eq:Krylov_decomposition}
 \cH\;=\;\overline{\mathcal{K}(A,g)}\,\oplus\,\mathcal{K}(A,g)^\perp
\end{equation}
is called the \emph{Krylov decomposition} of $\cH$ relative to $A$ and $g$.
  \item[(ii)] An operator $T\in\mathcal{B}(\cH)$ is said to be $\mathcal{K}(A,g)$-\emph{reduced} when both $\overline{\mathcal{K}(A,g)}$ and $\mathcal{K}(A,g)^\perp$ are invariant under $T$. Such a requirement, since $T$ is bounded, is equivalent to $T=T|_{\overline{\mathcal{K}(A,g)}}\oplus T|_{\mathcal{K}(A,g)^\perp}$, namely $T$ is reduced with respect to the Krylov decomposition \eqref{eq:Krylov_decomposition} \cite[Prop.~1.15]{schmu_unbdd_sa}.
  \item[(iii)] The subspace
  \begin{equation}\label{eq:Krilov_intersection}
   \mathcal{I}(A,g)\;:=\;\overline{\mathcal{K}(A,g)}\,\cap\, (A\,\mathcal{K}(A,g)^\perp) 
  \end{equation}
 is called the \emph{Krylov intersection} for the given $A$ and $g$. 
 \end{itemize}
\end{defn}

Krylov reducibility is inspired by the straightforward observation that
\begin{equation}\label{eq:invariances}
 A\,\mathcal{K}(A,g)\,\subset\,\mathcal{K}(A,g)\,,\qquad A^*\,\mathcal{K}(A,g)^\perp\,\subset\,\mathcal{K}(A,g)^\perp\,,
\end{equation}
whence also
\begin{equation}\label{eq:invariances2}
 A\,\overline{\mathcal{K}(A,g)}\,\subset\,\overline{\mathcal{K}(A,g)}\,.
\end{equation}
Thus, $\overline{\mathcal{K}(A,g)}$ is always $A$-invariant, and when so too is $\mathcal{K}(A,g)^\perp$ one says that $A$ is Krylov-reducible. Evidently, any bounded \emph{self-adjoint} operator $A$ is $\mathcal{K}(A,g)$-reduced and it is easy to construct non-self-adjoint  $A$'s that are $\mathcal{K}(A,g)$-reduced as well \cite[Example 2.3]{CMN-2018_Krylov-solvability-bdd}.

It is also clear that if $A$ is $\mathcal{K}(A,g)$-reduced, then in particular the Krylov intersection is trivial: $\mathcal{I}(A,g)=\{0\}$. That the converse is not true in general is easily seen already at the finite-dimensional level (with obvious infinite-dimensional generalisation) in \cite[Example 2.5]{CMN-2018_Krylov-solvability-bdd}.

A significant property is the following.

\begin{prop}\label{prop:KrylIntTriv_KrylSol} $($\cite[Prop.~3.4]{CMN-2018_Krylov-solvability-bdd}$.)$
  For a given Hilbert space $\cH$ let $A\in\mathcal{B}(\cH)$ be injective and $g\in\mathrm{ran} A$. Let $f\in\cH$ satisfy $Af=g$.
 \begin{itemize}
  \item[(i)] If $\mathcal{I}(A,g)=\{0\}$, then $f\in\overline{\mathcal{K}(A,g)}$.
  \item[(ii)] Assume further that $A$ is invertible with everywhere defined, bounded inverse on $\cH$. Then $f\in\overline{\mathcal{K}(A,g)}$ if and only if $\mathcal{I}(A,g)=\{0\}$.
 \end{itemize}
\end{prop}

$\mathcal{K}(A,g)$-reducibility of $A$ is a special case of triviality of $\mathcal{I}(A,g)$, and is therefore sufficient to ensure the Krylov solvability for $Af=g$. This is the case for any self-adjoint $A$, as already observed. More generally:

\begin{prop}\label{prop:whenA_Kreduced} $($\cite[Prop.~2.4]{CMN-2018_Krylov-solvability-bdd}$.)$
  For a given Hilbert space $\cH$ let $A\in\mathcal{B}(\cH)$ and $g\in\mathrm{ran} A$. Assume further that $A$ is normal. Then $A$ is $\mathcal{K}(A,g)$-reduced if and only if $A^*g\in\overline{\mathcal{K}(A,g)}$, in which case the associated inverse problem \eqref{eq:invLP} is Krylov-solvable.
\end{prop}

On the other hand, there are also inverse problems that are Krylov-solvable because they have a trivial Krylov intersection, \emph{without} being Krylov-reduced. An obvious example is the problem \cite[Example 2.5]{CMN-2018_Krylov-solvability-bdd}. Even though the operator in \cite[Example 2.5]{CMN-2018_Krylov-solvability-bdd} is \emph{not} normal, one can find analogous examples also in the relevant class of bounded, injective, \emph{normal} operators \cite[Example 3.8]{CMN-2018_Krylov-solvability-bdd}.

This discussion gives a strong evidence that the \emph{triviality of the Krylov intersection} is the correct mechanism that captures the emergence of Krylov solvability.

In fact, the triviality of the Krylov intersection ensures also the \emph{existence} of a Krylov solution.

\begin{prop} $($\cite[Prop.~3.4 and 3.9]{CMN-2018_Krylov-solvability-bdd}$.)$
 For a given Hilbert space $\cH$ let $A\in\mathcal{B}(\cH)$ and $g\in\mathrm{ran} A$. If $\mathcal{I}(A,g)=\{0\}$, then there exists $f\in\overline{\mathcal{K}(A,g)}$ such that $Af=g$.
\end{prop}

In turn, whereas not all bounded normal inverse problems are Krylov solvable, as seen above, \emph{normality} ensures that the Krylov solution, if existing, is unique.

\begin{prop} $($\cite[Prop.~3.10]{CMN-2018_Krylov-solvability-bdd}$.)$ 
 For a given Hilbert space $\cH$ let $A\in\mathcal{B}(\cH)$ and $g\in\mathrm{ran} A$. If $A$ is normal, then there exists at most one $f\in\overline{\mathcal{K}(A,g)}$ such that $Af=g$. More generally, the same conclusion holds if $A$ is bounded with $\ker A\subset\ker A^*$.  
\end{prop}

\begin{cor}\label{cor:self-adj_Kry}
 If $A\in\mathcal{B}(\cH)$ is self-adjoint, then the inverse problem $Af=g$ with $g\in\mathrm{ran}A$ admits a unique Krylov solution.
\end{cor}

\section{The positive self-adjoint case: conjugate gradients.}\label{sec:cg}

While in section \ref{sec:bounded} we surveyed our current knowledge of Krylov solvability for inverse problems induced by bounded operators, let us now enter the scenario that is the object of the present work, namely Krylov solvability for the problem \eqref{eq:invLP} when the operator $A$ is possibly \emph{unbounded}.

Prior to discussing the unbounded case in fairly wide generality (section \ref{sec:general_unbdd}), we find it instructive to analyse, in this and the following section, a distinguished class of unbounded inverse problems that are relevant in applications, the \emph{self-adjoint} inverse problems.

Of course, Corollary \ref{cor:self-adj_Kry} already provides a complete (and affirmative) answer on the question of Krylov solvability when $A$ is \emph{bounded and self-adjoint}. Therefore, although the discussion of this and the following section covers also the bounded case as well, the perspective is actually on the unbounded case.

In this section, in particular, we consider the inverse problem \eqref{eq:invLP} when $A$ is a (possibly unbounded) self-adjoint and non-negative operator: $A=A^*\geqslant\mathbb{O}$.

This is in fact an extremely relevant case in applications, for it is the setting of the ample class of popular Krylov-subspace-based algorithms for the numerical solution to \eqref{eq:invLP} collectively known as the `\emph{method of conjugate gradients}' (also referred to as CG). CG was first proposed in 1952 by Hestenes and Stiefel \cite{Hestenes-Stiefer-ConjGrad-1952} and since then, together with its related derivatives (e.g., conjugate gradients method on the normal equations (CGNE), least-square QR method (LSQR), etc.), it has been widely studied in the finite-dimensional setting (see the monographs \cite{Saad-2003_IterativeMethods,Simoncini-Szyld-2007,Liesen-Strakos-2003}) and also, though to a lesser extent, in the infinite-dimensional Hilbert space setting.

For the purposes of the present discussion, let us briefly recall what the algorithm consists of in the special version that most evidently manifests its nature of a Krylov subspace algorithm.

Associated to the inverse problem \eqref{eq:invLP}, with $A=A^*\geqslant\mathbb{O}$ and $g$ satisfying \eqref{eq:assumption-g}, one has the solution manifold
\begin{equation}
 \mathcal{S}(A,g)\;:=\;\{f\in\mathcal{D}(A)\,|\,Af=g\}\,.
\end{equation}
As $\mathcal{S}(A,g)$ is a convex, non-empty set in $\cH$ which is also closed, owing to the self-adjointness and hence closedness of $A$, the projection map $P_\mathcal{S}:\cH\to\mathcal{S}(A,g)$ is unambiguously defined and produces, for generic $x\in\cH$, the closest to $x$ point in $\mathcal{S}(A,g)$. The $A$-smoothness of $g$ makes the definition \eqref{eq:defKrylov} of the Krylov subspace $\mathcal{K}(A,g)$ well posed, and next to it one can also consider the $N$-th order subspaces
\begin{equation}\label{eq:defKrylovN}
  \mathcal{K}_N(A,g) \;:=\;  \mathrm{span} \{g,Ag,\dots,A^{N-1}g\}\,,\qquad N\in\mathbb{N}\,.
\end{equation}

The CG method, in the special version that we are reviewing now, then consists of producing `iterates' $f_N\in \mathcal{K}_N(A,g)$ by means of the minimisation
\begin{equation}\label{eq:iterates}
 f_N\;:=\;\:\arg\!\!\!\!\!\min_{\!\!\!\!\!\!\!\!\! h\in\mathcal{K}_N(A,g)}\big\langle (h-P_\mathcal{S}h),A(h-P_\mathcal{S}h)\big\rangle\,,\qquad N\in\mathbb{N}\,.
\end{equation}
The jargon here reminds us that there are implementations of CG, equivalent to \eqref{eq:iterates}, which produce the $f_N$'s \emph{iteratively}, with no reference to the a priori knowledge of $P_\mathcal{S}$, and hence clearly suited for numerics \cite{Saad-2003_IterativeMethods,Liesen-Strakos-2003}.

Clearly, if for some $N$ one has $f_N\in \mathcal{S}(A,g)$, then $A f_N=g$: the algorithm has come to convergence in a finite number of steps. This is the case when $\dim\cH<\infty$. When instead $\dim\cH=\infty$, the generic behaviour of the CG iterates is to get asymptotically closer and closer to the solution manifold, thus providing approximate solutions to the inverse problem \eqref{eq:invLP}.

It is worth mentioning that all iterates \eqref{eq:iterates} have the same projection onto the solution manifold, more precisely \cite[Prop.~2.1]{CM-Nemi-unbdd-2019},
\begin{equation}
 P_\mathcal{S}\,f_N\;=\;P_\mathcal{S}\,\mathbf{0}\;=:f^\circ\qquad\forall N\in\mathbb{N}\,,
\end{equation}
that is, for all the $f_N$'s their closest to $\mathcal{S}(A,g)$ point is the projection onto $\mathcal{S}(A,g)$ of the zero vector of $\cH$. Since, by linearity of $A$, $\mathcal{S}(A,g)$ is in fact an affine space, $f^\circ$ is the \emph{minimal norm solution} to $Af=g$.

When $\dim\cH=\infty$, the convergence theory $f_N\to f^\circ$ has been studied over the last five decades, both in the scenario where $A$ is bounded with everywhere-defined bounded inverse \cite{Daniel-1967,Daniel-1970,Herzog-Ekkehard-2015}, or at least with bounded inverse on its range \cite{Kammerer-Nashed-1972}, and in the scenario where $A$ is bounded with possible unbounded inverse on its range \cite{Kammerer-Nashed-1972,Nemirovskiy-Polyak-1985,Nemirovskiy-Polyak-1985-II,Louis-1987,Brakhage-1987,Hanke-ConjGrad-1995,Engl-Hanke-Neubauer-1996}. Recently, the general unbounded-$A$ case was covered too \cite{CM-Nemi-unbdd-2019}. For concreteness, when $g$ is a quasi-analytic vector (recall that the set $\mathcal{D}^{qa}(A)$ of quasi-analytic vectors for $A$ are dense in $\cH$, \cite[Sect. 7.4]{schmu_unbdd_sa}) one has the following.

\begin{thm}\label{thm:cg}
 For a given Hilbert space $\cH$ let  $A=A^*\geqslant\mathbb{O}$ and let $g\in\mathrm{ran} A\cap \mathcal{D}^{qa}(A)$. Then the inverse problem $Af=g$ is Krylov solvable, and in particular one has
 \begin{equation}
  \lim_{N\to\infty}\big\| f_N-f^\circ\big\|\;=\;0
 \end{equation}
 along the sequence of the conjugate gradient iterates $f_N\in\mathcal{K}_N(A,g)$ defined in \eqref{eq:iterates}, where $f^\circ$ is the minimal norm solution to the considered inverse problem. 
\end{thm}

Theorem \ref{thm:cg} is the special case of a much wider class of convergence results for CG that for the (non-negative, self-adjoint) bounded-$A$ case were proved in full completeness by Nemirovskiy and Polyak \cite{Nemirovskiy-Polyak-1985,Nemirovskiy-Polyak-1985-II}, and for the (non-negative, self-adjoint) unbounded-$A$ case were proved in our recent work \cite{CM-Nemi-unbdd-2019}. (For the reader's reference, Theorem \ref{thm:cg} is the special case of \cite[Theorem 2.4]{CM-Nemi-unbdd-2019} when, in the notation therein, $f^{[0]}=\mathbf{0}$, $\sigma=0$, $\xi=1$.)

This establishes Krylov solvability in the framework of unbounded, non-negative, self-adjoint inverse problems.

\section{The general self-adjoint and skew-adjoint case}\label{sec:selfadj_skewadj}

In this section we present our first main result, that in practice extends Corollary \ref{cor:self-adj_Kry} to the whole classes of (possibly unbounded) self-adjoint or skew-adjoint operators.

As such, this also generalises the conjugate-gradient-based Krylov solvability statement of Theorem \ref{thm:cg} established for (possibly unbounded) non-negative, self-adjoint inverse problems. The reason why we dealt first with the CG-analysis of section \ref{sec:cg} is that our next Theorem \ref{thm:self-skew} is in fact based on the special non-negative case of Theorem \ref{thm:cg}.

\begin{thm}\label{thm:self-skew}
 For a given Hilbert space $\cH$ let  $A$ be a self-adjoint ($A^*=A$) or skew-adjoint ($A^*=-A)$ operator on $\cH$ and let $g\in\mathrm{ran} A\cap \mathcal{D}^{qa}(A)$. Then there exists a unique solution $f$ to $Af=g$ such that $f\in\overline{\mathcal{K}(A,g)}$. Thus, the inverse problem $Af=g$ is Krylov-solvable.
\end{thm}

\begin{proof} \underline{Existence}. In the self-adjoint case, since $A^2$ is self-adjoint and $A^2\geqslant \mathbb{O}$, Theorem \ref{thm:cg} implies that there exists $f\in\overline{\mathcal{K}(A^2,Ag)}\subset\overline{\mathcal{K}(A,g)}$ such that $A^2f=Ag$. Analogously, in the skew-adjoint case, since $-A^2$ is self-adjoint and $-A^2\geqslant \mathbb{O}$, then there exists $f\in\overline{\mathcal{K}(-A^2,-Ag)}=\overline{\mathcal{K}(A^2,Ag)}\subset\overline{\mathcal{K}(A,g)}$ such that $-A^2f=-Ag$.
 
 In either case, $f\in\overline{\mathcal{K}(A,g)}$, $f\in\mathcal{D}(A^2)\subset\mathcal{D}(A)$, and $A^2f=Ag$, equivalently, $A(Af-g)=0$. This shows that $Af-g\in\ker A$.
 
 On the other hand, both $Af$ and $g$ belong to $\mathrm{ran}A$, whence $Af-g\in\mathrm{ran}A\subset(\ker A^*)^\perp=(\ker A)^\perp$, where the last identity is clearly valid both for in the self-adjoint and in the skew-adjoint case.
 
 Then necessarily $Af-g=0$, which proves that $f\in\overline{\mathcal{K}(A,g)}$ is a solution to the considered inverse problem.
 
 \underline{Uniqueness}. If $f_1,f_2\in\overline{\mathcal{K}(A,g)}$ and $Af_1=g=Af_2$, then $f_1-f_2\in\ker A\cap \overline{\mathcal{K}(A,g)}$. Moreover, $\ker A=\ker A^*$ and $\overline{\mathcal{K}(A,g)}\subset\overline{\mathrm{ran}A}$. Therefore, $f_1-f_2\in\ker A^*\cap \overline{\mathrm{ran}A}$. Thus, $f_1=f_2$.
\end{proof}

\begin{rem}
 In view of the discussion in section \ref{sec:cg}, the proof of Theorem \ref{thm:self-skew} shows that the actual Krylov-solution $f$ to $Af=g$ is the minimal norm solution to $A^2f=Ag$ and admits approximants $f_N$, with $\|f_N-f\|_{\cH}\to 0$ as $N\to\infty$, defined by
 \[
   f_N\;:=\;\;\arg\!\!\!\!\!\!\min_{\!\!\!\!\!\!\!\!\! h\in\mathcal{K}_N(A^2,Ag)}\big\|A(h-f)\big\|_{\cH}^2\,.
 \]
 Thus, the iterates of the CG algorithm applied to the auxiliary problem $A^2f=Ag$ (interpreted as $-A^2f=-Ag$ in the skew-adjoint case) converge precisely to the Krylov solution to $Af=g$. 
\end{rem}

\begin{rem}
 Unbounded skew-adjoint inverse problems are intimately related to inverse problems induced by so-called Friedrichs operators. These constitute a class of elliptic, parabolic, and hyperbolic differential operators, that can be also characterised as abstract operators on Hilbert space $\cH$ \cite{EGC,ABcpde,ABCE,antonic_erceg_michelangeli_jde2017}, having the typical (but not the only one) form $T=A+C$ where $A^*=-A$ and $C\in\mathcal{B}(\cH)$. Theorem \ref{thm:self-skew} is applicable when $C$ is skew-adjoint itself.
\end{rem}

\section{New phenomena in the general unbounded case: `Krylov escape', generalised Krylov reducibility, generalised Krylov intersection}\label{sec:general_unbdd}

Let us start in this section the analysis of Krylov solvability of the inverse problem \eqref{eq:invLP}, under the working condition \eqref{eq:gsmooth}, when the (possibly unbounded) operator $A$ is \emph{densely defined} and \emph{closed} in $\cH$ -- without necessarily being self-adjoint.

A number of substantial novelties, due to domain issues, emerge in this case as compared to the bounded case discussed in section \ref{sec:bounded}.

The first unavoidable difference concerns the invariance of $\overline{\mathcal{K}(A,g)}$ (respectively, $\mathcal{K}(A,g)^\perp$) under the action of $A$ (resp., of $A^*$). Indeed, the inclusions \eqref{eq:invariances} certainly cannot be valid in general, because the above subspaces may well not be included, respectively, in $\mathcal{D}(A)$ and $\mathcal{D}(A^*)$.

\begin{ex}\label{example:Acreator}
 The `quantum mechanical creation operator'
 \[
  \begin{split}
   A\;&=\;-\frac{\ud}{\ud x}+x \\
   \mathcal{D}(A)\;&=\;\big\{h\in L^2(\mathbb{R})\,|\,-h'+xh\in L^2(\mathbb{R})\big\}
  \end{split}
 \]
 is densely defined, unbounded, and closed, and has the well-known property that
 \[
  \psi_{n+1}\;=\;\frac{1}{\sqrt{n+1}}\,A\psi_n\qquad n\in\mathbb{N}_0\,,
 \]
 where $(\psi_n)_{n\in\mathbb{N}_0}$ is the orthonormal basis of $L^2(\mathbb{R})$ of the Hermite functions $\psi_n(x)=c_n H_n(x) e^{-x^2/2}$ (here $c_n$ is a normalisation factor and $H_n$ is the $n$-th Hermite polynomial). In particular, each $\psi_n$ is a $C^\infty(A)$-function. Choosing $g=\psi_1$ evidently yields $\overline{\mathcal{K}(A,g)}=\mathrm{span}\{\psi_0\}^\perp$. But there are $L^2$-functions orthogonal to $\psi_0$ that do not belong to $\mathcal{D}(A)$. 
\end{ex}

It is then clear that only the possible invariance of $\overline{\mathcal{K}(A,g)}\cap\mathcal{D}(A)$ under $A$ and of $\mathcal{K}(A,g)^\perp\cap\mathcal{D}(A^*)$ under $A^*$ makes sense in general.

This naturally leads one to consider the operators $A|_{\overline{\mathcal{K}(A,g)}\cap\mathcal{D}(A)}$ (the so-called `part of $A$ on $\overline{\mathcal{K}(A,g)}$') and $A^*|_{\mathcal{K}(A,g)^\perp\cap\mathcal{D}(A^*)}$ (the `part of $A^*$ on $\mathcal{K}(A,g)^\perp$'). Noticeably, when $A$ is unbounded (and hence $\mathcal{D}(A)$ is a proper dense subspace of $\cH$), none of the two is densely defined in $\cH$, unless $\overline{\mathcal{K}(A,g)}=\cH$, as their domain is by construction the intersection of a proper dense and a proper closed subspace. Obviously, instead,  $A|_{\overline{\mathcal{K}(A,g)}\cap\mathcal{D}(A)}$ is densely defined in the Hilbert space $\overline{\mathcal{K}(A,g)}$.

\begin{lem}
 For a given Hilbert space $\cH$ let  $A$ be a densely defined operator on $\cH$ and let $g\in C^\infty(A)$. Then
 \begin{equation}\label{eq:Astarinv}
  A^*\big( \mathcal{K}(A,g)^\perp\cap\mathcal{D}(A^*) \big) \;\subset\; \mathcal{K}(A,g)^\perp\,.
 \end{equation}
\end{lem}

\begin{proof}
Let $z\in \mathcal{K}(A,g)^\perp\cap\mathcal{D}(A^*)$. For arbitrary $h\in \overline{\mathcal{K}(A,g)}$ let $(h_n)_{n\in\mathbb{N}}$ be a sequence in $\mathcal{K}(A,g)$ of norm-approximants of $h$. Then each $Ah_n\in \mathcal{K}(A,g)$, and therefore
\[
 \langle h, A^*z\rangle\;=\;\lim_{n\to\infty}\langle h_n, A^*z\rangle\;=\;\lim_{n\to\infty}\langle Ah_n,z\rangle\;=\;0\,,
\]
thus proving \eqref{eq:Astarinv}.
\end{proof}

The counterpart inclusion to \eqref{eq:Astarinv}, namely $A(\overline{\mathcal{K}(A,g)}\cap\mathcal{D}(A))\subset\overline{\mathcal{K}(A,g)}$ when $\overline{\mathcal{K}(A,g)}$ is only a proper closed subspace of $\cH$, turns out to be considerably less trivial. In fact, as somewhat counterintuitive as it appears, $A$ may indeed map vectors from $\overline{\mathcal{K}(A,g)}\cap\mathcal{D}(A)$ \emph{outside} of $\overline{\mathcal{K}(A,g)}$. In the present context, we shall refer to this phenomenon, that has no analogue in the bounded case, as `\emph{Krylov escape}'.

\begin{ex}[Krylov escape]
 Let $\cH'$ be a Hilbert space and $T'$ be a self-adjoint operator in $\cH'$ having a cyclic vector $g'$, meaning that there exists $g'\in\mathcal{D}(T')$ such that $\overline{{K}(T',g')}=\cH'$. (It is straightforward to construct many explicit examples for such a choice.) For any $1$-dimensional vector space $\cH_0$, say, $\cH_0=\mathrm{span}\{e_0\}$, set
\[
 \begin{split}
  \cH\;&:=\;\cH_0\oplus\cH' \\
  T\;&:=\;\mathbb{O}\oplus T' \\
  g\;&:=\;\mathbf{0}\oplus g'\,.
 \end{split}
\]
(The last condition is just an identification of $g$ as an element of $\cH$.)
Thus, $T$ is a self-adjoint operator in $\cH$ such that $Te_0=\mathbf{0}$ and $Tx'=T'x'$ $\forall x'\in\mathcal{D}(T')$, and moreover $\overline{\mathcal{K}(T,g)}=\cH'$. (Now the closure is taken with respect to $\cH$.) Furthermore,  let $x_0\in \cH$ such that $x_0\in \cH'\setminus\mathcal{D}(T')$. Then set 
\[
 \begin{split}
  \mathcal{D}(A)\;&:=\;\mathcal{D}(T)\dotplus\mathrm{span}\{x_0\} \\
  Ax_0\;&:=\;e_0 \\
  Ax\;&:=\;Tx\qquad\forall x\in\mathcal{D}(T)\,.
 \end{split}
\]
$A$ is meant to be defined by the above identities and extended by linearity on the whole $\mathcal{D}(T)\dotplus\mathrm{span}\{x_0\}$. The operator $A$ is densely defined in $\cH$ by construction.

$\bullet$ \underline{Closedness}. Let us check that if $x_n+\mu_n x_0\to v$ and $A(x_n+\mu_n x_0)\to w$ in $\cH$ as $n\to\infty$ for some vectors $v,w\in \cH$, where $(x_n)_{n\in\mathbb{N}}$ is a generic sequence in $\mathcal{D}(T)$ and $(\mu_n)_{n\in\mathbb{N}}$ is a generic sequence in $\mathbb{C}$, then $v\in\mathcal{D}(A)$ and $Av=w$. First, we observe that it must be $\mu_n\to\mu$ for some $\mu\in\mathbb{C}$, for otherwise there would be no chance for the vectors $A(x_n+\mu_n x_0)=Tx_n+\mu_n e_0$ to converge as assumed, because $Tx_n\perp\mu_n e_0$. Next, since $\mu_n\to\mu$ and $x_n+\mu_n x_0\to v$, then necessarily $x_n\to x$ for some $x\in\cH$ satisfying $x+\mu x_0=v$; analogously, since  $Tx_n+\mu_n e_0=A(x_n+\mu_n x_0)\to w$, then $Tx_n\to y$ for some $y\in\cH$ satisfying $y+\mu e_0=w$. As by construction $T$ is self-adjoint and hence closed, then necessarily $x\in\mathcal{D}(T)$ and $Tx=y$. In turn, this implies that $v\in\mathcal{D}(A)$ and $Av=w$. The conclusion is that $A$ is closed. 

$\bullet$ \underline{Occurrence of Krylov escape}. By the construction above, $\overline{\mathcal{K}(A,g)}=\cH'=\mathrm{span}\{e_0\}^\perp$. Let us focus on the vector $e_0$. On the one hand, $e_0=Ax_0$ and $x_0$ by definition belongs both to $\overline{\mathcal{K}(A,g)}$ and to $\mathcal{D}(A)$. Therefore $e_0\in A(\overline{\mathcal{K}(A,g)}\cap\mathcal{D}(A))$. On the other hand, however, $e_0\in(\cH')^\perp=\mathcal{K}(A,g)^\perp$, whence $e_0\notin\overline{\mathcal{K}(A,g)}$. This provides a counterexample of a densely defined closed operator $A$ in $\cH$ such that the inclusion
\[
 A\big(\overline{\mathcal{K}(A,g)}\cap\mathcal{D}(A)\big)\;\subset\;\overline{\mathcal{K}(A,g)}\,.
\]
is \emph{violated}.
\end{ex}

Owing to the possible occurrence of the Krylov escape phenomenon, the invariance of $\overline{\mathcal{K}(A,g)}\cap\mathcal{D}(A)$ requires additional assumptions. A reasonable one is to assume further that the operator $A|_{\overline{\mathcal{K}(A,g)}\cap\mathcal{D}(A)}$ and its restriction $A|_{\mathcal{K}(A,g)}$ are in a sense as close as possible. To this aim, let us observe first the following.

\begin{lem}\label{lem:Arestrictedclosure}
  For a given Hilbert space $\cH$ let  $A$ be a densely defined and closed operator on $\cH$ and let $g\in C^\infty(A)$. Then
  \begin{itemize}
   \item[(i)] the operator $A|_{\overline{\mathcal{K}(A,g)}\cap\mathcal{D}(A)}$ is closed,
   \item[(ii)] and the operator $A|_{\mathcal{K}(A,g)}$ is closable.
  \end{itemize}
\end{lem}

\begin{proof} Obviously $A|_{\mathcal{K}(A,g)}\subset A|_{\overline{\mathcal{K}(A,g)}\cap\mathcal{D}(A)}$, in the sense of operator inclusion, so part (ii) follows at once from part (i). In turn, part (i) is true as is always the case when one restricts a \emph{closed} operator to the intersection of its domain with a closed subspace: the restriction operator too is closed. Explicitly, let $((x_n,Ax_n))_{n\in\mathbb{N}}$ be an arbitrary $\cH\oplus\cH$-convergent sequence in the graph of $A|_{\overline{\mathcal{K}(A,g)}\cap\mathcal{D}(A)}$, that is, for some $x,y\in\cH$ one has $\overline{\mathcal{K}(A,g)}\cap\mathcal{D}(A)\ni x_n\to x$ and $Ax_n\to y$ in $\cH$. Then, by closedness of $A$, $x\in\mathcal{D}(A)$ and $Ax=y$. Moreover, since $x_n\in\overline{\mathcal{K}(A,g)}$ $\forall n$, also for the limit point one has $x\in\overline{\mathcal{K}(A,g)}$. Thus, $x\in\overline{\mathcal{K}(A,g)}\cap\mathcal{D}(A)$. This shows that the pair $(x,y)$ belongs to the graph of $A|_{\overline{\mathcal{K}(A,g)}\cap\mathcal{D}(A)}$, which is therefore closed in $\cH\oplus\cH$. 
\end{proof}

\begin{rem}
 With a completely analogous argument one shows that the operator $A^*|_{\mathcal{K}(A,g)^\perp\cap\mathcal{D}(A^*)}$ is closed too.
\end{rem}

It is then natural to consider the case when the operator closure of $A|_{\mathcal{K}(A,g)}$ is precisely $A|_{\overline{\mathcal{K}(A,g)}\cap\mathcal{D}(A)}$.

\begin{defn}\label{def:Krylovcorecond}
 For a given Hilbert space $\cH$ let  $A$ be a densely defined and closed operator on $\cH$ and let $g\in C^\infty(A)$. Then the pair $(A,g)$ is said to satisfy the `\emph{Krylov-core condition}' when the subspace $\mathcal{K}(A,g)$ is a core for $A|_{\overline{\mathcal{K}(A,g)}\cap\mathcal{D}(A)}$. Explicitly, this is the requirement that 
 \begin{equation}\label{eq:Krylovcorecond}
 \overline{A|_{\mathcal{K}(A,g)}}\;=\;A|_{\overline{\mathcal{K}(A,g)}\cap\mathcal{D}(A)}
\end{equation}
 in the sense of operator closure, equivalently, it is the requirement that $\mathcal{K}(A,g)$ is dense in $\overline{\mathcal{K}(A,g)}\cap\mathcal{D}(A)$ in the graph norm $\|h\|_A:=(\|h\|_{\cH}^2+\|Ah\|_{\cH}^2)^{\frac{1}{2}}$:
 \begin{equation}\label{eq:Krylovcorecond2}
  \overline{\mathcal{K}(A,g)}^{\|\,\|_A}\;=\;\overline{\mathcal{K}(A,g)}\cap\mathcal{D}(A)\,.
 \end{equation}
\end{defn}

\begin{rem}\label{rem:one-inclusion-is-trivial}
 By closedness of $A$, the inclusion
 \begin{equation}\label{eq:Krylovcorecond2-oneinclusion}
  \overline{\mathcal{K}(A,g)}^{\|\,\|_A}\;\subset\;\overline{\mathcal{K}(A,g)}\cap\mathcal{D}(A)
 \end{equation}
 is always  true, as one sees reasoning as for Lemma \ref{lem:Arestrictedclosure}.
\end{rem}

\begin{ex}
 Let $\cH=\ell^2(\mathbb{N}_0)$ and, in terms of the canonical orthonormal basis $(e_n)_{n\in\mathbb{N}_0}$, let $A$ be the densely defined and closed operator defined by
 \[
 \begin{split}
   \mathcal{D}(A)\;&:=\;\Big\{x\equiv(x_n)_{n\in\mathbb{N}_0}\,\Big|\,\sum_{n=0}^\infty (n+1)^2|x_n|^2<+\infty\Big\} \\
   Ax\;&:=\;(0,x_0,2x_1,3x_2,\dots)\qquad\textrm{for }x\in\mathcal{D}(A)
 \end{split}
 \]
 (thus, in particular, $Ae_n=(n+1) e_{n+1}$ for any $n\in\mathbb{N}_0$). 
 Obviously, we have $\mathcal{K}(A,e_0)=\mathrm{span}\{e_1,e_2,e_3,\dots\}$ and $\overline{\mathcal{K}(A,e_0)}=\{e_0\}^\perp$. Let
  \[
   x\;:=\;(0,x_1,x_2,x_3,x_4,\dots)\quad\textrm{with}\quad\sum_{n=1}^\infty (n+1)^2|x_n|^2<+\infty
 \]
 be a generic vector in $\overline{\mathcal{K}(A,e_0)}\cap\mathcal{D}(A)$ and for each $N\in\mathbb{N}$ let
 \[
   x^{(N)}\;:=\;(0,x_1^{(N)},\dots,x_N^{(N)},0,0,\dots)\quad\textrm{with}\quad x_n^{(N)}=x_n+\frac{1}{n^2N}\,.
 \]
 Then $x^{(N)}\in\mathcal{K}(A,e_0)$ and
 \[
  \begin{split}
   \big\| x-x^{(N)}\big\|_A^2\;&=\;\sum_{n=1}^\infty (1+(n+1)^2)\big|x_n-x_n^{(N)}\big|^2 \\
   &=\;\frac{1}{\,N^2}\sum_{n=1}^N\frac{n^2+2n+2}{n^4}+\sum_{n=N+1}^\infty(1+(n+1)^2)|x_n|^2 \xrightarrow[]{N\to\infty}\;0\,,
  \end{split}
 \]
 whence $x\in\overline{\mathcal{K}(A,e_0)}^{\|\,\|_A}$. Thus, $\overline{\mathcal{K}(A,e_0)}^{\|\,\|_A}\supset\overline{\mathcal{K}(A,e_0)}\cap\mathcal{D}(A)$, which together with \eqref{eq:Krylovcorecond2-oneinclusion} shows that the pair $(A,e_0)$ does satisfy the Krylov-core condition \eqref{eq:Krylovcorecond2}. 
\end{ex}

The Krylov-core condition is indeed sufficient to finally ensure that $A$ maps $\overline{\mathcal{K}(A,g)}\cap\mathcal{D}(A)$ into $\overline{\mathcal{K}(A,g)}$.

\begin{lem}\label{lem:krycore-implies-noescape}
 For a given Hilbert space $\cH$ let  $A$ be a densely defined and closed operator on $\cH$ and let $g\in  C^\infty(A)$.
 \begin{itemize}
  \item[(i)] One has the inclusion
  \begin{equation}\label{eq:invarianceKrylovcorecond}
  A\big( \overline{\mathcal{K}(A,g)}\cap\mathcal{D}(A)\big)\;\subset\;\overline{A\mathcal{K}(A,g)}
 \end{equation}
 if and only if $A$ and $g$ satisfy the Krylov-core condition \eqref{eq:Krylovcorecond}.
 \item[(ii)] In particular, under the Krylov-core condition one has
 \begin{equation}\label{eq:AKinv}
  A\big( \overline{\mathcal{K}(A,g)}\cap\mathcal{D}(A)\big)\;\subset\; \overline{\mathcal{K}(A,g)}\,.
 \end{equation}
 \end{itemize}
\end{lem}

\begin{proof}
 Clearly \eqref{eq:AKinv} follows at once from \eqref{eq:invarianceKrylovcorecond} as $\overline{A\mathcal{K}(A,g)}\subset  \overline{\mathcal{K}(A,g)}$. Let us focus then on the proof of part (i). Assume that \eqref{eq:Krylovcorecond} is satisfied and let $z\in\overline{\mathcal{K}(A,g)}\cap\mathcal{D}(A)$, the domain of $A|_{\overline{\mathcal{K}(A,g)}\cap\mathcal{D}(A)}$. Since the latter operator is the closure of $A|_{\mathcal{K}(A,g)}$, then there exists $(z_n)_{n\in\mathbb{N}}$ in $\mathcal{K}(A,g)$ such that $z_n\to z$ and $A|_{\mathcal{K}(A,g)}z_n\to A|_{\overline{\mathcal{K}(A,g)}\cap\mathcal{D}(A)} z$, i.e., $Az_n\to Az$. This shows that $Az\in\overline{A\mathcal{K}(A,g)}$. For the converse implication, assume now that $z\in\overline{\mathcal{K}(A,g)}\cap\mathcal{D}(A)$ and that $A|_{\overline{\mathcal{K}(A,g)}\cap\mathcal{D}(A)}$ is a \emph{proper} closed extension of $\overline{A|_{\mathcal{K}(A,g)}}$. This means that whatever sequence $(z_n)_{n\in\mathbb{N}}$ in $\mathcal{K}(A,g)$ of norm-approximants of $z$ is considered, one cannot have $Az_n\to Az$, because this would mean $\overline{A|_{\mathcal{K}(A,g)}}z_n\to A|_{\overline{\mathcal{K}(A,g)}\cap\mathcal{D}(A)} z$, that is, $\overline{A|_{\mathcal{K}(A,g)}}z=A|_{\overline{\mathcal{K}(A,g)}\cap\mathcal{D}(A)} z$, contrarily to the assumption $ \overline{A|_{\mathcal{K}(A,g)}}\varsubsetneq A|_{\overline{\mathcal{K}(A,g)}\cap\mathcal{D}(A)}$. Thus, $Az$ cannot have norm-approximants in $A\mathcal{K}(A,g)$ and hence \eqref{eq:invarianceKrylovcorecond} cannot be valid. 
\end{proof}

\begin{rem}\label{rem:AKinK}
 In general $\overline{A\mathcal{K}(A,g)}\varsubsetneq\overline{\mathcal{K}(A,g)}$. Example \ref{example:Acreator} shows that in that case $\overline{A\mathcal{K}(A,g)}=\mathrm{span}\{\psi_0,\psi_1\}^\perp$, whereas $\overline{\mathcal{K}(A,g)}=\mathrm{span}\{\psi_0\}^\perp$.
\end{rem}

%
%
%
%
%

In the framework of the preceding discussion the two key notions of Krylov reducibility and Krylov intersection introduced in the bounded case in Definition \ref{def:Kreduc-Kintersec-bounded} can be now generalised to the unbounded case.

\begin{defn}\label{def:Kreduc-Kintersec-unbounded}
  For a given Hilbert space $\cH$ let  $A$ be a densely defined and closed operator on $\cH$ and let $g\in  C^\infty(A)$.
  \begin{itemize}
   \item $A$ is said to be $\mathcal{K}(A,g)$-\emph{reduced (in the generalised sense)}, for short \emph{Krylov-reduced}, when 
   \begin{equation}\label{eq:Krylov-reducibility-generalised}
    \begin{split}
     A\big( \overline{\mathcal{K}(A,g)}\cap\mathcal{D}(A)\big)\;&\subset\; \overline{\mathcal{K}(A,g)} \\
     A\big( \mathcal{K}(A,g)^\perp\cap\mathcal{D}(A)\big)\;&\subset\; \mathcal{K}(A,g)^\perp\,. 
    \end{split}
   \end{equation}
   \item[(ii)] The subspace
   \begin{equation}\label{eq:Krilov_intersection-generalised}
    \mathcal{I}(A,g)\;:=\;\overline{\mathcal{K}(A,g)}\,\cap\, A\big( \mathcal{K}(A,g)^\perp\cap\mathcal{D}(A)\big)
   \end{equation}
   is called the \emph{(generalised) Krylov intersection} for the given $A$ and $g$. 
  \end{itemize}
\end{defn}

As was the case for Definition \ref{def:Kreduc-Kintersec-bounded}, it is clear from \eqref{eq:Krylov-reducibility-generalised} and \eqref{eq:Krilov_intersection-generalised} that also in the generalised sense of Definition \ref{def:Kreduc-Kintersec-unbounded} Krylov reducibility implies triviality of the Krylov intersection.

\begin{rem}\label{rem:Kcore-implies-no-escape}
 The first condition of the two requirements \eqref{eq:Krylov-reducibility-generalised} for Krylov-re\-du\-ci\-bility is precisely the lack of Krylov escape \eqref{eq:AKinv}. Thus this condition is matched if, for instance, $A$ and $g$ satisfy the Krylov-core condition (Lemma \ref{lem:krycore-implies-noescape}).
\end{rem}

\begin{rem}
 Krylov reducibility in the generalised sense \eqref{eq:Krylov-reducibility-generalised} for an \emph{unbounded} operator differs from Krylov reducibility in the \emph{bounded} case, which is formulated as
   \begin{equation*}
    \begin{split}
     A\overline{\mathcal{K}(A,g)}\;&\subset\; \overline{\mathcal{K}(A,g)} \\
     A\mathcal{K}(A,g)^\perp\;&\subset\; \mathcal{K}(A,g)^\perp\,,
    \end{split}
   \end{equation*}
 in that when $A$ is bounded the subspaces $\overline{\mathcal{K}(A,g)}$ and $\mathcal{K}(A,g)^\perp$ are \emph{reducing} for $A$ and hence with respect to the Krylov decomposition $\cH= \overline{\mathcal{K}(A,g)}\oplus \mathcal{K}(A,g)^\perp$ the operator $A$ decomposes as $A=A|_{\overline{\mathcal{K}(A,g)}}\oplus A|_{\mathcal{K}(A,g)^\perp}$ whereas if $A$ is unbounded and Krylov-reduced it is false in general that $A=A|_{\overline{\mathcal{K}(A,g)}\cap\mathcal{D}(A)}\oplus A|_{\mathcal{K}(A,g)^\perp\cap\mathcal{D}(A)}$.
\end{rem}

%

\section{Krylov solvability in the general unbounded case}\label{sec:genunbdd}

In this section we examine counterparts of Krylov solvability of the inverse problem \eqref{eq:invLP} when the operator $A$ is densely defined and closed.

A crucial role in this matter turns out to be played by the lack of Krylov escape, namely the property \eqref{eq:AKinv}, a feature that was automatically present in the bounded case. A first example is the following technical Lemma that will be useful in a moment.

\begin{lem}\label{lem:fK_AKdenseK}
 For a given Hilbert space $\cH$ let $A$ be a densely defined and closed operator on $\cH$, let $g\in\mathrm{ran}(A)\cap C^\infty(A)$, and let $f\in\mathcal{D}(A)$ satisfy $Af=g$. If in addition the pair $(A,g)$ satisfies the Krylov-core condition and $f\in\overline{\mathcal{K}(A,g)}$, then
 \begin{equation}\label{eq:noKrylovEscape}
  \overline{A \mathcal{K}(A,g)}\;=\;\overline{\mathcal{K}(A,g)}\,.
 \end{equation} 
\end{lem}

\noindent(We already observed in Remark \ref{rem:AKinK} that in general $\overline{A \mathcal{K}(A,g)}\varsubsetneq\overline{\mathcal{K}(A,g)}$.)

\begin{proof}[Proof of Lemma \ref{lem:fK_AKdenseK}]
 By assumption, $f\in \overline{\mathcal{K}(A,g)}\cap\mathcal{D}(A)$, which is the same as $f\in \overline{\mathcal{K}(A,g)}^{\| \|_A}$, owing to the Krylov-core condition. Thus there exists a sequence $(f_n)_{n\in\mathbb{N}}$ in $\mathcal{K}(A,g)$ such that $f_n\xrightarrow[]{\|\,\|_A}f$ as $n\to\infty$, in particular $Af_n\xrightarrow[]{\|\,\|}Af=g$, which implies that $g$ belongs to $\overline{A\mathcal{K}(A,g)}$. Clearly all vectors $Ag,A^2g,A^3g,\dots$ belong to the same space too. Therefore,
 \[
  \mathrm{span}\{A^kg\,|\,k\in\mathbb{N}_0\}\;\subset\;\overline{A\mathcal{K}(A,g)}\,,
 \]
 so that $\overline{\mathcal{K}(A,g)}\subset\overline{A\mathcal{K}(A,g)}$. The opposite inclusion is trivial, hence \eqref{eq:noKrylovEscape} follows. 
\end{proof}

For convenience we shall denote by $P_{\mathcal{K}}$ the orthogonal projection $P_{\mathcal{K}}:\cH\to\cH$ onto the Krylov subspace $\overline{\mathcal{K}(A,g)}$.

In the unbounded injective case, the triviality of the Krylov intersection still implies Krylov solvability under additional assumptions that are automatically satisfied if $A$ is bounded. The first requires that the orthogonal projection on $\overline{\mathcal{K}(A,g)}$ of the solution manifold $\mathcal{S}(A,g)$ is entirely contained in the domain of $A$. The other requirement is the lack of Krylov escape.

In fact, Proposition \ref{prop:KrylIntTriv_KrylSol} admits the following analogues.

\begin{prop}\label{prop:KrylIntTriv_KrylSol_unbdd}
  For a given Hilbert space $\cH$ let $A$ be a densely defined and closed operator on $\cH$, let $g\in\mathrm{ran}(A)\cap C^\infty(A)$, and let $f\in\mathcal{D}(A)$ satisfy $Af=g$. Assume furthermore that
  \begin{itemize}
   \item[(a)] $A$ is injective;
   \item[(b)] $A(\overline{\mathcal{K}(A,g)}\cap\mathcal{D}(A))\subset \overline{\mathcal{K}(A,g)}$ -- this assumption holds true, for example, if the pair $(A,g)$ satisfies the Krylov-core condition (Lemma \ref{lem:krycore-implies-noescape});
   \item[(c)] $P_{\mathcal{K}}f\in\mathcal{D}(A)$;
   \item[(d)] $\mathcal{I}(A,g)=\{0\}$,
  \end{itemize}
  or also, assume the more stringent assumptions
  \begin{itemize}
   \item[(a)] $A$ is injective;
   \item[(b')] $A$ is $\mathcal{K}(A,g)$-reduced
    \item[(c)] $P_{\mathcal{K}}f\in\mathcal{D}(A)$.
  \end{itemize}
 Under such assumptions, $f\in\overline{\mathcal{K}(A,g)}$.
\end{prop}

\begin{proof} By assumption (c), $P_\mathcal{K}f\in\overline{\mathcal{K}(A,g)}\cap\mathcal{D}(A)$; by assumption (b), $AP_\mathcal{K}f\in\overline{\mathcal{K}(A,g)}$. Then $A(\mathbbm{1}-P_\mathcal{K})f=g-AP_\mathcal{K}f\in\overline{\mathcal{K}(A,g)}$. On the other hand, again by assumption (c), $(\mathbbm{1}-P_\mathcal{K})f\in\mathcal{D}(A)$, whence $A(\mathbbm{1}-P_\mathcal{K})f\in A( \mathcal{K}(A,g)^\perp\cap\mathcal{D}(A))$. Thus, $A(\mathbbm{1}-P_\mathcal{K})f\in\mathcal{I}(A,g)$. By assumptions (a) and (d), then $f=P_\mathcal{K}f$.
\end{proof}

\begin{prop}
  For a given Hilbert space $\cH$ let $A$ be a densely defined and closed operator on $\cH$, let $g\in\mathrm{ran}(A)\cap C^\infty(A)$, and let $f\in\mathcal{D}(A)$ satisfy $Af=g$. Assume furthermore that
  \begin{itemize}
   \item[(a)] $A$ is invertible with everywhere defined, bounded inverse on $\cH$;
   \item[(b)] the pair $(A,g)$ satisfies the Krylov-core condition.
  \end{itemize}
 Under such assumptions, if $f\in\overline{\mathcal{K}(A,g)}$, then $\mathcal{I}(A,g)=\{0\}$.
\end{prop}

\begin{proof}
 Let $z\in\mathcal{I}(A,g)$. Then $z=Aw$ for some $w\in\mathcal{K}(A,g)^\perp\cap\mathcal{D}(A)$, and $z\in\overline{\mathcal{K}(A,g)}$. Owing to Lemma \ref{lem:fK_AKdenseK}, $\overline{\mathcal{K}(A,g)}=\overline{A\mathcal{K}(A,g)}$, hence there is a sequence $(v_n)_{n\in\mathbb{N}}$ in $\mathcal{K}(A,g)$ such that $Av_n\to z=Aw$ as $n\to\infty$. Then also $v_n\to w$, because $\|v_n-w\|_{\cH}\leqslant\|A^{-1}\|_{\mathrm{op}}\|Av_n-z\|_{\cH}$. Since $v_n\perp w$ for each $n$, then
 \[
  0\;=\;\lim_{n\to\infty}\|v_n-w\|_{\cH}^2\;=\;\lim_{n\to\infty}\big(\|v_n\|_{\cH}^2+\|w\|_{\cH}^2\big)\;=\;2\|w\|_{\cH}^2\,,
 \]
 whence $w=0$ and also $z=Aw=0$. 
\end{proof}

When $A$ is not injective, from the perspective of the Krylov-solvability of the inverse problem \eqref{eq:invLP} one immediately makes two observations. First, if $g=0$, then trivially $\mathcal{K}(A,g)=\{0\}$ and therefore the Krylov space does not capture any of the non-zero solutions to \eqref{eq:invLP}, which all belong to $\ker A$. Second, if $g\neq 0$ and therefore the problem of Krylov-solvability is non-trivial, it is natural to ask whether a Krylov solution exists and whether it is unique.

The following two Propositions provide answers to these questions.

\begin{prop}
  For a given Hilbert space $\cH$ let $A$ be a densely defined operator on $\cH$ and  let $g\in\mathrm{ran}(A)\cap C^\infty(A)$.  If $\ker A\subset\ker A^*$ (in particular, if $A$ is normal), then there exists at most one $f\in\overline{\mathcal{K}(A,g)}\cap\mathcal{D}(A)$ such that $Af=g$. 
%
\end{prop}

 \begin{proof}
  The argument is analogous to that used for the corresponding statement of Theorem \ref{thm:self-skew}.
  If $f_1,f_2\in\overline{\mathcal{K}(A,g)}$ and $Af_1=g=Af_2$, then $f_1-f_2\in\ker A\cap \overline{\mathcal{K}(A,g)}$. By assumption, $\ker A\subset\ker A^*$, and  moreover obviously $\overline{\mathcal{K}(A,g)}\subset\overline{\mathrm{ran}A}$. Therefore, $f_1-f_2\in\ker A^*\cap \overline{\mathrm{ran}A}=\{0\}$, whence $f_1=f_2$.
 \end{proof}

 \begin{prop}\label{prop:existence_Krylov_sol}
  For a given Hilbert space $\cH$ let $A$ be a densely defined and closed operator on $\cH$, let $g\in\mathrm{ran}(A)\cap C^\infty(A)$, and let $f\in\mathcal{D}(A)$ satisfy $Af=g$. Assume furthermore that
  \begin{itemize}
   \item[(a)] $A(\overline{\mathcal{K}(A,g)}\cap\mathcal{D}(A))\subset \overline{\mathcal{K}(A,g)}$;
   \item[(b)] $P_{\mathcal{K}}f\in\mathcal{D}(A)$;
   \item[(c)] $\mathcal{I}(A,g)=\{0\}$,
  \end{itemize}
  or also, assume the more stringent assumptions
  \begin{itemize}
   \item[(a')] $A$ is $\mathcal{K}(A,g)$-reduced
    \item[(b')] $P_{\mathcal{K}}f\in\mathcal{D}(A)$.
  \end{itemize}
 Then there exists $f_\circ\in\overline{\mathcal{K}(A,g)}\cap\mathcal{D}(A)$ such that $Af_\circ=g$.  
%
%
\end{prop}

  \begin{proof} Let $f$ be a solution to $Af=g$ (it certainly exists, for $\mathcal{S}(A,g)$ is non-empty). Reasoning as in the proof of Proposition \ref{prop:KrylIntTriv_KrylSol_unbdd}: $P_\mathcal{K}f\in\overline{\mathcal{K}(A,g)}\cap\mathcal{D}(A)$ (by (b)), $AP_\mathcal{K}f\in\overline{\mathcal{K}(A,g)}$ (by (a)), $(\mathbbm{1}-P_\mathcal{K})f\in\mathcal{D}(A)$ (by (b)), whence $A(\mathbbm{1}-P_\mathcal{K})f\in A( \mathcal{K}(A,g)^\perp\cap\mathcal{D}(A))$ and also $A(\mathbbm{1}-P_\mathcal{K})f=g-AP_\mathcal{K}f\in\overline{\mathcal{K}(A,g)}$. Thus, $A(\mathbbm{1}-P_\mathcal{K})f\in\mathcal{I}(A,g)$, whence $g=Af=AP_\mathcal{K}f$ (by (c)). Set
  \[
   f_\circ\;:=\;P_\mathcal{K}f\;\in\;\overline{\mathcal{K}(A,g)}\cap\mathcal{D}(A)\,.
  \]
  Now, if $g=0$, then $\overline{\mathcal{K}(A,g)}=\{0\}$, whence $f_\circ=0$: this (trivial) solution to the corresponding inverse problem is indeed a Krylov-solution. If instead $g\neq 0$, then necessarily $f_\circ\neq 0$. In either case $Af_\circ=g$.  
\end{proof}

\section{The self-adjoint case revisited: structural properties.}\label{sec:selfadj_revisited}

We return in this section to the general question of Krylov-solvability for an inverse problem of the form $Af=g$ when $A=A^*$.

More precisely, whereas the analysis of section~\ref{sec:selfadj_skewadj} (Theorem \ref{thm:self-skew}) already provides an affirmative answer, based on conjugate gradient arguments, it does not explain how the operator $A$ and the datum $g$ behave in the self-adjoint case with respect to the abstract operator-theoretic mechanisms for Krylov-solvability identified in section~\ref{sec:genunbdd} (Krylov-reducibility, triviality of the Krylov intersection, stability of the solution manifold inside $\mathcal{D}(A)$ under the projection $P_\mathcal{K}$).

Of course it is straightforward to observe that if $A=A^*$ and $g\in C^\infty(A)$, then the second of the two conditions \eqref{eq:Krylov-reducibility-generalised} for Krylov reducibility is automatically true (owing to \eqref{eq:Astarinv}) and therefore $\mathcal{I}(A,g)$ is always trivial.

Unlike the bounded case, however, in order for $A$ to be $\mathcal{K}(A,g)$-reduced no Krylov escape must occur, namely $A$ has to match also the first of the two conditions \eqref{eq:Krylov-reducibility-generalised}, and we have already observed (Remark \ref{rem:Kcore-implies-no-escape}) that an assumption on $A$ and $g$ such as the Krylov-core condition would indeed prevent the Krylov escape phenomenon.

As relevant such issues are, from a more abstract perspective, to understand why `structurally' a self-adjoint inverse problem is Krylov-solvable, and as deceptively simple the underlying mathematical questions appear, to our knowledge no complete answer is available in the affirmative (i.e., a proof) or in the negative (a counterexample) to the following questions:
\begin{itemize}
 \item[(\textbf{Q1})] When $A=A^*$ and $g\in C^\infty(A)$, is it true that $\overline{\mathcal{K}(A,g)}^{\|\,\|_A}=\overline{\mathcal{K}(A,g)}\cap\mathcal{D}(A)$, i.e., is the Krylov core condition satisfied by the pair  $(A,g)$ ?
  \item[(\textbf{Q2})] When $A=A^*$ and $g\in C^\infty(A)$, is it true that $ A\big( \overline{\mathcal{K}(A,g)}\cap\mathcal{D}(A)\big)\subset\overline{\mathcal{K}(A,g)}$, i.e., is $A$ $\mathcal{K}(A,g)$-reduced in the generalised sense?
\end{itemize}
(Clearly in (\textbf{Q2}) it is tacitly understood that $\mathcal{K}(A,g)$ is not dense in $\cH$.)

We can provide a partial answer in a vast class of cases, namely whenever the vector $g$ is `bounded' for $A$.

Let us recall (see, e.g., \cite[Sect.~7.4]{schmu_unbdd_sa}) that a $C^\infty$-vector $g\in\cH$ for a linear operator $A$ on a Hilbert space $\cH$ is \emph{bounded} for $A$ when there is a constant $B_g>0$ such that 
\begin{equation}\label{eq:boundedvectors}
 \|A^n g\|_{\cH}\;\leqslant\;B_g^n\qquad\forall n\in\mathbb{N}_0\,.
\end{equation}
As well-known (see, e.g., \cite[Lemma 7.13]{schmu_unbdd_sa}), the vector space of bounded vectors for $A$, when $A$ is self-adjoint, is dense in $\cH$.

\begin{thm}\label{thm:krylov-core_analytic}
 For a given Hilbert space $\cH$ let $A$ be a self-adjoint operator on $\cH$ and let $g\in\cH$ be a bounded vector for $A$. Then the pair $(A,g)$ satisfies the Krylov-core condition and consequently $A$ is $\mathcal{K}(A,g)$-reduced in the generalised sense of Definition \ref{def:Kreduc-Kintersec-unbounded}. 
\end{thm}

 As observed already (Remark \ref{rem:one-inclusion-is-trivial}), the thesis follows from the inclusion
 \begin{equation}
  \overline{\mathcal{K}(A,g)}^{\|\,\|_A}\;\supset\;\overline{\mathcal{K}(A,g)}\cap\mathcal{D}(A)
 \end{equation}
  that we shall prove now.
 
\begin{proof}[Proof of Theorem \ref{thm:krylov-core_analytic}]
 Let $x\in\overline{\mathcal{K}(A,g)}\cap\mathcal{D}(A)$. We need to exhibit a sequence $(x_n)_{n\in\mathbb{N}}$ in $\mathcal{K}(A,g)$ such that $x_n\xrightarrow[]{\;\|\,\|_A\;}x$ as $n\to\infty$. Since $x\in\overline{\mathcal{K}(A,g)}$, there surely exists a sequence $(p_n(A)g)_{n\in\mathbb{N}}$, for some polynomials $p_n$, that converges to $x$  in the $\cH$-norm, although a priori not in the stronger $\|\,\|_A$-norm.

 First we show that one can refine and `regularise' such a sequence with respect to the action of $A$ in order to strengthen the convergence. Explicitly, we show that up to taking a subsequence of the $p_n$'s, one has
 \[\tag{i}
  e^{-A^2/(nB_g)}p_n(A)g\;\xrightarrow[n\to\infty]{\;\|\,\|_A\;}\;x\,,
 \]
 where $B_g$ is the constant of \eqref{eq:boundedvectors}.
 
 To this aim, we split
 \[
  \begin{split}
   \big\|x- e^{-A^2/(nB_g)}p_n(A)g\big\|_A\;&\leqslant\;\big\|x- e^{-A^2/(nB_g)}x\big\|_A \\
   &\qquad +\big\|e^{-A^2/(nB_g)}(x-p_n(A)g)\big\|_A
  \end{split}
 \]
 and observe that by dominated convergence
  \[
   \big\|x- e^{-A^2/(nB_g)}x\big\|_A^2\;=\;\int_\mathbb{R}(1 + \lambda^2)\,|1-e^{-\lambda^2/(nB_g)}|^2\,\ud\mu_{x}^{(A)}(\lambda)\;\xrightarrow[]{n\to\infty}\;0\,,
  \]
 where $\mu_{x}^{(A)}$ is the scalar spectral measure for $A$ relative to the vector $x$. A suitable integrable majorant for the dominated convergence argument is $4(1 +\lambda^2)$ and for that we used the assumption that $x\in\mathcal{D}(A)$. As for the second summand, one has
 \[
  \begin{split}
   \big\|e^{-A^2/(nB_g)}(x-p_n(A)g)\big\|_A^2\;&=\;\big\|(\mathbbm{1}+A^2)^{\frac{1}{2}}e^{-A^2/(nB_g)}(x-p_n(A)g)\big\|_{\cH}^2 \\
   &\leqslant\;\big\|(\mathbbm{1}+A^2)^{\frac{1}{2}}e^{-A^2/(nB_g)}\big\|_{\mathrm{op}}^2\,\|x-p_n(A)g\|_{\cH}^2\,.
  \end{split}
 \]
By assumption $\|x-p_n(A)g\|_{\cH}\xrightarrow{n\to\infty} 0$, whereas $\|(\mathbbm{1}+A^2)^{\frac{1}{2}}e^{-A^2/(nB_g)}\|_{\mathrm{op}}$ $=\|(1+\lambda)^{1/2}e^{-\lambda^2/(nB_g)}\|_{L^\infty(\mathbb{R}^+)} < + \infty$ diverges in the limit if $A$ is unbounded. Up to passing to a convenient subsequence of the $p_n$'s, which we shall denote again by $(p_n)_{n\in\mathbb{N}}$, it is always possible to make the vanishing of $\|x-p_n(A)g\|_{\cH}$ sufficiently fast so as to compensate the divergence of $\|(\mathbbm{1}+A^2)^{\frac{1}{2}}e^{-A^2/(nB_g)}\|_{\mathrm{op}}$ and make their product eventually vanish.

 Clearly the new sequence $(e^{-A^2/(nB_g)}p_n(A)g)_{n\in\mathbb{N}}$ has the expected $\|\,\|_A$-con\-ver\-gence to $x$, but a priori does not belong to $\mathcal{K}(A,g)$ any longer. Now we claim that there exists a monotone increasing sequence of integers $(N_n)_{n\in\mathbb{N}}$, with $N_n\to\infty$ as $n\to\infty$, such that the vectors
 \[
  x_n\;:=\;\sum_{k=0}^{N_n}\frac{(-1)^k}{\,k!(nB_g)^k}\, A^{2k}\, p_n(A)\,g\;\in\;\mathcal{K}(A,g)
 \]
 satisfy
 \[\tag{ii}
  \big\|e^{-A^2/(nB_g)}p_n(A)g-x_n\big\|_A\;\xrightarrow[]{n\to\infty}\;0
 \]
 whence also, combining (i) and (ii), the thesis $x_n\xrightarrow[]{\;\|\,\|_A\;}x$.

 To prove the claim, for each $n\in\mathbb{N}$ and $\lambda\in\mathbb{R}$ we use the notation
 \[
  p_n(\lambda)\;=\;\sum_{\ell=0}^{D_n} a_\ell^{(n)} \lambda^\ell\,,
 \]
 where $a_\ell^{(n)}\in\mathbb{C}$ and $D_n\in\mathbb{N}$ with $a_{D_n}^{(n)}\neq 0$ (meaning that $\mathrm{deg}\,p_n=D_n$). Let us focus on the second summand in the r.h.s.~of the identity
 \[
 \begin{split}
  \big\|e^{-A^2/(nB_g)}p_n(A)g-x_n\big\|_A^2\;&=\;\big\|e^{-A^2/(nB_g)}p_n(A)g-x_n\big\|_{\cH}^2 \\
  &\qquad +\big\|A(e^{-A^2/(nB_g)}p_n(A)g-x_n)\big\|_{\cH}^2
 \end{split}
 \]
 (the argument for the first summand is the very same), and let us re-write
 \[
  A\big(e^{-A^2/(nB_g)}p_n(A)g-x_n\big)\;=\;\sum_{\ell=0}^{D_n}a_\ell^{(n)}\Big(e^{-A^2/(nB_g)}-\sum_{k=0}^{N_n}\frac{(-1)^k}{k!}\Big(\frac{A^2}{nB_g}\Big)^k \Big)   A^{\ell+1}g
  \]
  for some $N_n$ to be determined.
  
 Thus,
 \[
  \begin{split}
   \big\|A\big(&e^{-A^2/(nB_g)}p_n(A)g-x_n\big)\big\|_{\cH} \\
   &\leqslant\;\sum_{\ell=0}^{D_n}\big|a_\ell^{(n)}\big|\,\Big\|\Big(e^{-A^2/(nB_g)}-\sum_{k=0}^{N_n}\frac{(-1)^k}{k!}\Big(\frac{A^2}{nB_g}\Big)^k \Big)   A^{\ell+1}g\Big\|_{\cH}\,.
  \end{split}
 \]

  Now, for each $\ell$ one has 
  \[
   \begin{split}
    \Big\|\Big(e^{-A^2/(nB_g)}&-\sum_{k=0}^{N_n}\frac{(-1)^k}{k!}\Big(\frac{A^2}{nB_g}\Big)^k \Big)   A^{\ell+1}g\Big\|_{\cH}^2 \\
    &=\;\int_{\mathbb{R}}\Big| e^{-\lambda^2/(nB_g)}-\sum_{k=0}^{N_n}\frac{(-1)^k}{k!}\Big(\frac{\lambda^2}{nB_g}\Big)^k \Big|^2   \lambda^{2(\ell+1)}\,\ud\mu_{g}^{(A)}(\lambda) \\
    &=\;\int_{\mathbb{R}}\Big| \sum_{k=N_n+1}^{\infty}\frac{(-1)^k}{k!}\Big(\frac{\lambda^2}{nB_g}\Big)^k \Big|^2   \,\lambda^{2(\ell+1)}\,\ud\mu_{g}^{(A)}(\lambda) \\
    &\leqslant\;\int_{\mathbb{R}}\Big|\frac{1}{(N_n+1)!}\Big(\frac{\lambda^2}{nB_g}\Big)^{\!N_n+1}\Big|^2\,\lambda^{2(\ell+1)}\,\ud\mu_{g}^{(A)}(\lambda) \\
    &=\;\frac{1}{((N_n+1)!\,(nB_g)^{N_n+1})^2}\,\|A^{N_n+\ell+2}g\|_{\cH}^2 \\
    &\leqslant\;\Big(\frac{B_g^{\ell+1}}{(N_n+1)!\,n^{N_n+1}}\Big)^2\,,
   \end{split}
  \]
  where the last inequality above is due to assumption \eqref{eq:boundedvectors}, whereas the previous inequality follows by a standard estimate of the remainder in Taylor's formula.

  Combining the last two estimates one finds
  \[
    \big\|A\big(e^{-A^2/(nB_g)}p_n(A)g-x_n\big)\big\|_{\cH}\;\leqslant\;\sum_{\ell=0}^{D_n}\big|a_\ell^{(n)}\big|\,\frac{B_g^{\ell+1}}{(N_n+1)!\,n^{N_n+1}}\,.
  \]
  Each $\ell$-th term of the sum above depends on $n$. We can argue that there is a suitable sequence $(N_n)_{n\in\mathbb{N}}$ in $\mathbb{N}$, with $N_n\to \infty$ as $n\to\infty$, such that for every $\ell\in\mathbb{N}_0$ and for every $n\in\mathbb{N}$ with $n\geqslant \max\{1,B_g\}$ one has
  \[
   \big|a_\ell^{(n)}\big|\,\frac{B_g^{\ell+1}}{(N_n+1)!\,n^{N_n+1}}\;\leqslant\;\frac{1}{n\,(\ell+1)^2}\,.
  \]
  Indeed, let $(N_n)_{n\in\mathbb{N}}$ be a generic monotone increasing and divergent sequence of natural numbers and let us refine it as follows. First, let us take a subsequence, for convenience denoted again by $(N_n)_{n\in\mathbb{N}}$, such that $N_n\geqslant D_n$. Thus,
  \[
    \big|a_\ell^{(n)}\big|\,\frac{B_g^{\ell+1}}{(N_n+1)!\,n^{N_n+1}}\;\leqslant\;\frac{\big|a_\ell^{(n)}\big|}{(N_n+1)!}\qquad\textrm{for }n\geqslant\max\{1,B_g\}\,.
  \]
  Next, for each fixed $\ell\in\mathbb{N}_0$, proceeding from $\ell=0$ and increasing $\ell$ by one at each step, let us further refine $(N_n)_{n\in\mathbb{N}}$ to a subsequence, renamed $(N_n)_{n\in\mathbb{N}}$ again, such that $N_n\geqslant n(\ell+1)^2|a_\ell^{(n)}|-1$. Then
  \[
   \frac{\big|a_\ell^{(n)}\big|}{(N_n+1)!}\;\leqslant\;\frac{\big|a_\ell^{(n)}\big|}{N_n+1}\;\leqslant\;\frac{1}{n(\ell+1)^2}\qquad\textrm{for }n\geqslant\max\{1,B_g\}\,.
  \]
  The above procedure amounts to a countable number of refinements of the original sequence $(N_n)_{n\in\mathbb{N}}$, and thus produces a final subsequence, still denoted by $(N_n)_{n\in\mathbb{N}}$, with the desired property, and for which therefore (for $n\geqslant\max\{1,B_g\}$)
   \[
   \begin{split}
      \big\|A\big(e^{-A^2/(nB_g)}p_n(A)g-x_n\big)\big\|_{\cH}\;&\leqslant\;\sum_{\ell=0}^{D_n}\frac{1}{n(\ell+1)^2}\;\leqslant\;\frac{1}{n}\sum_{\ell=0}^\infty\frac{1}{\,(\ell+1)^2} \\
      &=\;\frac{\,\pi^2}{6n}\;\xrightarrow[]{\;n\to\infty\;}0\,.
   \end{split}
  \]

  This establishes (ii) and concludes the proof. 
\end{proof}

What seems to suggest a special relevance of the assumption that $g$ be a bounded vector for $A$ is the fact that not only under such assumption is the conclusion of Theorem \ref{thm:krylov-core_analytic} valid, but also the corresponding Krylov space is naturally isomorphic to $L^2(\mathbb{R},\ud\mu_g^{(A)})$.

\begin{thm}\label{thm:Krylov_isomorphism}
 For a given Hilbert space $\cH$ let $A$ be a self-adjoint operator on $\cH$ and let $g\in\cH$ be a bounded vector for $A$. Then $L^2(\mathbb{R},\ud\mu_g^{(A)})\cong\overline{\mathcal{K}(A,g)}$ via the isomorphism $h\mapsto h(A)g$.
\end{thm}

In fact, Theorem \ref{thm:Krylov_isomorphism} (that we shall prove in a moment) allows one to re-obtain Krylov solvability for the inverse problem \eqref{eq:invLP} when $A$ is self-adjoint and injective.

\begin{cor}
 For a given Hilbert space $\cH$ let $A$ be an injective and self-adjoint operator on $\cH$ and let $g\in\mathrm{ran} A$ be a bounded vector for $A$. Then the unique solution $f\in\mathcal{D}(A)$ such that $Af=g$ is a Krylov solution.
\end{cor}

\begin{proof}
 Let $h$ be the measurable function defined by $h(\lambda):=\lambda^{-1}$ for $\lambda\in\mathbb{R}$. Since $\|f\|_{\cH}^2=\int_\mathbb{R}\lambda^{-2}\ud\mu_g^{(A)}(\lambda)$, then $h\in L^2(\mathbb{R},\ud\mu_g^{(A)})$. Therefore, owing to Theorem \ref{thm:Krylov_isomorphism}, $f=h(A)g\in\overline{\mathcal{K}(A,g)}$. 
\end{proof}

For the proof of Theorem \ref{thm:Krylov_isomorphism} it is convenient to single out the following facts.

 \begin{lem}\label{lem:fL2fent}
  Let $\mu$ be a positive finite measure on $\mathbb{R}$. Then the space of $\mathbb{R}\to\mathbb{C}$ functions that are restrictions of $\mathbb{C}\to\mathbb{C}$ entire functions and are square-integrable with respect to $\mu$ is dense in $L^2(\mathbb{R},\ud \mu)$.
 \end{lem}
 
 \begin{proof}
  Let $f\in L^2(\mathbb{R},\ud\mu)$. For every $\varepsilon>0$ then there exists a continuous and $L^2(\mathbb{R},\ud \mu)$-function $f_{\mathrm{c}}$ such that $\|f-f_{\mathrm{c}}\|_{L^2(\mathbb{R},\ud \mu)}\leqslant\varepsilon$ (see, e.g., \cite[Chapt.~3, Theorem 3.14]{Rudin-realcomplexanalysis}). In turn, by Carleman's theorem (see, e.g., \cite[Chapt.~4, \S 3, Theorem 1]{Gaier_ComplexApproximation}), there exists an entire function $f_{\mathrm{e}}$ such that $|f_{\mathrm{c}}(\lambda)-f_{\mathrm{e}}(\lambda)|\leqslant \mathcal{E}(\lambda)$ $\forall\lambda\in\mathbb{R}$, where $\mathcal{E}$ is an arbitrary error function. $\mathcal{E}$ can be therefore chosen so as $\|f_{\mathrm{c}}-f_{\mathrm{e}}\|_{L^2(\mathbb{R},\ud \mu)}\leqslant\varepsilon$. This shows that $f_{\mathrm{e}}\in L^2(\mathbb{R},\ud \mu)$ and $\|f-f_{\mathrm{e}}\|_{L^2(\mathbb{R},\ud \mu)}\leqslant 2\varepsilon$. As $\varepsilon$ is arbitrary, this completes the proof.  
 \end{proof}

\begin{lem}\label{lem:fentire_fL2}
 For a given Hilbert space $\cH$ let $A$ be a self-adjoint operator on $\cH$ and let $g\in\cH$ be a bounded vector for $A$. If $f:\mathbb{C}\to\mathbb{C}$ is an entire function, then its restriction to the real line belongs to $L^2(\mathbb{R},\ud\mu_g^{(A)})$ and $g\in\mathcal{D}(f(A))$.
\end{lem}

\begin{proof}
 As $f$ is entire, in particular $f(\lambda)=\sum_{k=0}^\infty\frac{f^{(k)}(0)}{k!}\lambda^k$ for every $\lambda\in\mathbb{R}$, where the series converges point-wise for every $\lambda$ and uniformly on compact subsets of $\mathbb{R}$. Thus, by functional calculus (see, e.g., \cite[Prop.~4.12(v)]{schmu_unbdd_sa}), for every $N\in\mathbb{N}$
 \[
  \mathbf{1}_{[-N,N]}(A)f(A)g\;=\;\sum_{k=0}^\infty\frac{f^{(k)}(0)}{k!}(A\mathbf{1}_{[-N,N]}(A))^k g\,,
 \]
 where $\mathbf{1}_\Omega$ denotes the characteristic function of $\Omega$ and the series in the r.h.s.~above converges in the norm of $\cH$. Again by standard properties of the functional calculus one then has
 \[
  \begin{split}
   \Big(\int_{-N}^N|f(\lambda)|^2\,\ud\mu_g^{(A)}(\lambda)\Big)^{\!1/2}\;&=\;\big\|\mathbf{1}_{[-N,N]}(A)f(A)g\big\|_{\cH} \\
   &=\;\Big\| \sum_{k=0}^\infty\frac{f^{(k)}(0)}{k!}\,\mathbf{1}_{[-N,N]}(A)\,A^k g \Big\|_{\cH} \\
   &\leqslant\;\sum_{k=0}^\infty \Big|\frac{f^{(k)}(0)}{k!} \Big|\,\big\|\mathbf{1}_{[-N,N]}(A)\,A^k g\big\|_{\cH} \\
   &\leqslant\;\sum_{k=0}^\infty \Big|\frac{f^{(k)}(0)}{k!} \Big|\,\big\|A^k g\big\|_{\cH} \\
   &\leqslant\;M_R\sum_{k=0}^\infty \Big(\frac{B_g}{R}\Big)^k\;=\;M_R\big(1-B_g/R)^{-1}\,,
  \end{split}
 \]
 where in the last inequality we used the bound $\|A^k g\|_{\cH}\leqslant B_g^k$ for some $B_g>0$ and we applied Cauchy's estimate $|\frac{f^{(k)}(0)}{k!}|\leqslant M_R/R^k$ for $R>B_g$ and $M_R:=\max_{|z|=R}|f(x)|$. Taking the limit $N\to\infty$ yields the thesis.
\end{proof}

\begin{lem}\label{lem:fAisKrylov}
 For a given Hilbert space $\cH$ let $A$ be a self-adjoint operator on $\cH$ and let $g\in\cH$ be a bounded vector for $A$. Let $f:\mathbb{C}\to\mathbb{C}$ be an entire function and let $f(z)=\sum_{k=0}^\infty\frac{f^{(k)}(0)}{k!} z^k$, $z\in\mathbb{C}$, be its Taylor expansion. Then, 
 \[
  \lim_{n\to\infty}\Big\|f(A)g-\sum_{k=0}^n\frac{f^{(k)}(0)}{k!}\,A^k g\Big\|_{\cH}\;=\;0
 \]
 and therefore $f(A)g\in\overline{\mathcal{K}(A,g)}$.
\end{lem}

\begin{proof}
 Both $f$ and $\mathbb{C}\ni z\mapsto \sum_{k=0}^n\frac{f^{(k)}(0)}{k!}\,z^k$ are entire functions for each $n\in\mathbb{N}$, and so is their difference. Then, reasoning as in the proof of Lemma \ref{lem:fentire_fL2},
  \[
  \begin{split}
   \Big\|f(A)g-\sum_{k=0}^n\frac{f^{(k)}(0)}{k!}\,A^k g\Big\|_{\cH}\;&=\;\Big\| \sum_{k=n+1}^\infty\frac{f^{(k)}(0)}{k!}\,A^k g \Big\|_{\cH} \\
   &\leqslant\;\sum_{k=n+1}^\infty \Big|\frac{f^{(k)}(0)}{k!} \Big|\,\big\|A^k g\big\|_{\cH} \\
   &\leqslant\;M_R\sum_{k=n+1}^\infty \Big(\frac{B_g}{R}\Big)^k \\
   &=\;M_R(B_g/R)^{n+1}\big(1-B_g/R\big)^{-1}\;\xrightarrow[]{n\to\infty}\;0\,,
  \end{split}
 \]
 which completes the proof.
\end{proof}

Let us finally prove Theorem \ref{thm:Krylov_isomorphism}.

\begin{proof}[Proof of Theorem \ref{thm:Krylov_isomorphism}]
 Let us denote by $\mathscr{E}(\mathbb{R})$ the space of $\mathbb{R}\to\mathbb{C}$ functions that are restrictions of $\mathbb{C}\to\mathbb{C}$ entire functions. Owing to Lemma \ref{lem:fentire_fL2}, $\mathscr{E}(\mathbb{R})\subset L^2(\mathbb{R},\ud \mu_g^{(A)})$, and owing to Lemma \ref{lem:fL2fent}, $\mathscr{E}(\mathbb{R})$ is actually dense in $L^2(\mathbb{R},\ud \mu_g^{(A)})$. Moreover, for each $h\in\mathscr{E}(\mathbb{R})$, $h(A)g\in\overline{\mathcal{K}(A,g)}$, as found in Lemma \ref{lem:fAisKrylov}. As $\|h(A)g\|_{\cH}=\|h\|_{L^2(\mathbb{R},\ud \mu_g^{(A)})}$, the map $h\mapsto h(A)g$ is an isometry from $\mathscr{E}(\mathbb{R})$ to $\overline{\mathcal{K}(A,g)}$, which extends canonically to a $L^2(\mathbb{R},\ud \mu_g^{(A)})\to\overline{\mathcal{K}(A,g)}$ isometry. In fact, this map is also surjective and therefore is an isomorphism. Indeed, for a generic $w\in\overline{\mathcal{K}(A,g)}$ there exists a sequence $(p_n)_{n\in\mathbb{N}}$ of polynomials such that $p_n(A)g\to w$ in $\cH$ as $n\to\infty$, and moreover $(p_n)_{n\in\mathbb{N}}$ is a Cauchy sequence in $L^2(\mathbb{R},\ud \mu_g^{(A)})$ because $\|p_n-p_m\|_{L^2(\mathbb{R},\ud \mu_g^{(A)})}=\|p_n(A)g-p_m(A)g\|_{\cH}$ for any $n,m\in\mathbb{N}$. Therefore $p_n\to h$ in $L^2(\mathbb{R},\ud \mu_g^{(A)})$ for some $h$ and consequently $w=\lim_{n\to\infty}p_n(A)g=h(A)g$. 
\end{proof}

\begin{rem}
 The reasoning of the proof of Theorem \ref{thm:Krylov_isomorphism} reveals that for \emph{generic} $g\in C^\infty(A)$, not necessarily bounded for $A$, the map $h\mapsto h(A)g$ is an isomorphism
 \begin{equation}
  \overline{\{\,\textrm{polynomials }\mathbb{R}\to\mathbb{C}\,\}}^{\|\,\|_{L^2(\mathbb{R},\ud \mu_g^{(A)})}}\cong\;\overline{\mathcal{K}(A,g)}\,.
 \end{equation}
 The assumption of $A$-boundedness for $g$ was used to show, concerning the l.h.s.~above, that polynomials are dense in $L^2(\mathbb{R},\ud \mu_g^{(A)})$.
 \end{rem}


\section*{Acknowledgements}
A.~M.~gratefully acknowledges the support of the Alexander von Humboldt foundation.

\def\cprime{$'$}

\end{document}